\documentclass[12pt]{amsart}

\usepackage{amssymb,verbatim,amsmath,xcolor,mathrsfs,amsthm}

\usepackage[T1]{fontenc}
\usepackage[utf8]{inputenc}
\usepackage{csquotes}
\usepackage{euscript}
\usepackage{mathtools}
\usepackage[american]{babel}
\usepackage{caption}

\usepackage[inline]{enumitem}

\usepackage{centernot}

\usepackage{bm}

\usepackage{tikz-cd}

\usepackage{tikz}
\usetikzlibrary{intersections}

\usepackage[letterpaper,left=2.5cm,right=2.5cm,top=2.5cm,bottom=2.5cm,dvips]{geometry}
\usepackage[bookmarks=true,linktocpage=true,bookmarksdepth=1,pdfencoding=auto]{hyperref}

\usepackage{cleveref}[2012/02/15]
\crefformat{footnote}{#2\footnotemark[#1]#3}

\let\originalleft\left
\let\originalright\right
\renewcommand{\left}{\mathopen{}\mathclose\bgroup\originalleft}
\renewcommand{\right}{\aftergroup\egroup\originalright}

\DeclareFontFamily{OT1}{cmrx}{}
\DeclareFontShape{OT1}{cmrx}{m}{n}{<->cmr10}{}
\let\saveLongrightarrow\Longrightarrow
\makeatletter
\renewcommand*{\Longrightarrow}{%
    \mathrel{\rlap{\fontfamily{cmrx}\fontencoding{OT1}\selectfont=}%
    \hphantom{\saveLongrightarrow}%
    \llap{$\m@th\Rightarrow$}}}
\makeatother

\makeatletter
\newenvironment{proof_header}[1][\proofname]{\par
  \pushQED{\qed}%
  \normalfont\topsep6\p@\@plus6\p@\relax
  \trivlist
  \item[]
  {\itshape #1\@addpunct{.}}\hskip\labelsep\ignorespaces
}{%
  \popQED\endtrivlist\@endpefalse
}
\makeatother

\makeatletter
\patchcmd{\@bibitem}{\ignorespaces}{\label{bib-#1}\ignorespaces}{}{}
\makeatother

\newtheorem{thm}{Theorem}[section]

\newtheorem{lem}[thm]{Lemma}
\newtheorem{cor}[thm]{Corollary}

\theoremstyle{definition}

\newtheorem{rem}[thm]{Remark}

\newtheorem{ex}[thm]{Example}
\newtheorem{defn}[thm]{Definition}
\newtheorem*{defn*}{Definition}
\newtheorem{fact}[thm]{Fact}
\newtheorem{notation}[thm]{Notation}

\newtheorem*{construction*}{Construction}

\newtheorem{assumption}[thm]{Assumption}
\newtheorem*{perspective*}{Perspective}
\newtheorem*{motivation*}{Motivation}
\newtheorem*{summary*}{\normalfont\scshape Summary}
\newtheorem*{references*}{References}

\title[Quasi-homogeneity of the moduli space of stable maps (II)]{Quasi-homogeneity of the moduli space of stable maps to homogeneous spaces (II)}

\author{Christoph B\"arligea}
\address{Institut \'Elie Cartan de Lorraine (I\'ECL)\\ UMR~7502\\ FST\\ Campus Aiguillettes\\ B.P.~70239\\ 54506 Vand{\oe}uvre-l\`es-Nancy\\ France}
\email{christoph.baerligea@rub.de}
\thanks{The research was supported by the German Research Foundation (DFG) under the project number 345815019}
\date{December 27, 2018}

\begin{document}

\begin{abstract}

Let $G$ be a connected, simply connected, simple, complex, linear algebraic group. Let $P$ be an arbitrary parabolic subgroup of $G$. Let $X=G/P$ be the $G$-homogeneous projective space attached to this situation. Let $d\in H_2(X)$ be a degree. Let $\overline{M}_{0,3}(X,d)$ be the (coarse) moduli space of three pointed genus zero stable maps to $X$ of degree $d$. Building on and improving our previous results \cite{quasi-homogeneity-rev}, we prove that $\overline{M}_{0,3}(X,d)$ is quasi-homogeneous under the action of $\operatorname{Aut}(X)$ \emph{for all} minimal degrees $d$ in $H_2(X)$. By a minimal degree in $H_2(X)$, we mean a degree $d\in H_2(X)$ which is minimal with the property that $q^d$ occurs (with non-zero coefficient) in the quantum product $\sigma_u\star\sigma_v$ of two Schubert classes $\sigma_u$ and $\sigma_v$, where $\star$ denotes the product in the (small) quantum cohomology ring $\mathit{QH}^*(X)$ attached to $X$. Along the way, we prove that $\overline{M}_{0,3}(X,d)$ is quasi-homogeneous under the action of $G$ for all minimal degrees $d$ in $H_2(X)$ except for one instance of $G$, $P$ and $d$ which occurs in type $\mathsf{G}_2$.

\end{abstract}

\maketitle

\section{Introduction}

Let $G$ be a connected, simply connected, simple, complex, linear algebraic group. Let $P$ be a fixed but arbitrary parabolic subgroup of $G$. Let $X=G/P$ be the $G$-homogeneous projective space attached to this situation. We select once and for all a maximal torus $T$ and a Borel subgroup $B$ of $G$ such that
\[
T\subseteq B\subseteq P\subseteq G\,.
\]

We say that $d$ is a degree in $H_2(X)$ if $d$ is an effective homology class in $H_2(X)$. Let $d$ be a degree in $H_2(X)$. Let $\overline{M}_{0,3}(X,d)$ be the (coarse) moduli space of three pointed genus zero stable maps to $X$ of degree $d$. By definition, the moduli space $\overline{M}_{0,3}(X,d)$ parametrizes isomorphism classes $[C,p_1,p_2,p_3,\mu\colon C\to X]$ where:
\begin{itemize}
    \item 
    $C$ is a complex, projective, connected, reduced, (at worst) nodal curve of arithmetic genus zero.
    \item
    The marked points $p_i\in C$ are distinct and lie in the nonsingular locus.
    \item
    $\mu$ is a morphism such that $\mu_*[C]=d$.
    \item
    The pointed map $\mu$ has no infinitesimal automorphisms.
\end{itemize}

Basic properties of the moduli space $\overline{M}_{0,3}(X,d)$ can be found in \cite{pand}. It is a consequence of more general results in \cite{pand,kim}, namely of \cite[Theorem~2(i)]{pand} and \cite[Corollary~1]{kim}, that $\overline{M}_{0,3}(X,d)$ is a normal projective irreducible variety.

In this work, we ask the question if it is possible to prove stronger properties of $\overline{M}_{0,3}(X,d)$ than irreducibility. In fact, building on the work \cite{quasi-homogeneity-rev}, we completely solve the question of quasi-homogeneity under $G$/$\operatorname{Aut}(X)$ for minimal degrees in $H_2(X)$ in the sense of the following two definitions.

\begin{defn}

Let $d$ be a degree in $H_2(X)$. The natural action of $G$/$\operatorname{Aut}(X)$ on $X$ induces an action of $G$/$\operatorname{Aut}(X)$ on $\overline{M}_{0,3}(X,d)$ given by translation. We say that the moduli space $\overline{M}_{0,3}(X,d)$ is quasi-homogeneous under the action of $G$/$\operatorname{Aut}(X)$ if the action of $G$/$\operatorname{Aut}(X)$ on $\overline{M}_{0,3}(X,d)$ admits a dense open $G$-orbit/$\operatorname{Aut}(X)$-orbit.

\end{defn}

\begin{defn}
\label{def:minimal1}

Let $(\mathit{QH}^*(X),\star)$ be the (small) quantum cohomology ring attached to $X$ as defined in \cite[Section~10]{pand}. For a Weyl group element $w$, we denote by $\sigma_w$ the Schubert class associated to $w$.\footnote{For the definition of minimal degrees in $H_2(X)$, it does not matter if $\sigma_w$ is the class of a Schubert variety or opposite Schubert variety associated to $w$. To fix the ideas, the reader can think of $\sigma_w$ as the class of the opposite Schubert variety associated to $w$.} A degree $d\in H_2(X)$ is called a minimal degree in $H_2(X)$ if there exist Weyl group elements $u$ and $v$ such that $d$ is a minimal degree in $\sigma_u\star\sigma_v$, i.e.\ if the power $q^d$ occurs (with non-zero coefficient) in the expression $\sigma_u\star\sigma_v$ and if $d$ is minimal with this property, i.e.\ for all $d'<d$ the power $q^{d'}$ does not occur in the expression $\sigma_u\star\sigma_v$. For the meaning of \enquote{occurs in}, we refer to the \cite[beginning of Section~9]{fulton-woodward}. The partial order \enquote{$\leq$} on $H_2(X)$ is defined by the set of positive elements given by effective homology classes in $H_2(X)$.

\end{defn}

In \cite[Definition~7.3]{quasi-homogeneity-rev}, we constructed for each minimal degree $d$ in $H_2(X)$ a morphism $f_{P,d}\colon\mathbb{P}_\mathbb{C}^1\to X$ which lies in $\overline{M}_{0,3}(X,d)$. We recall the construction of $f_{P,d}$ in Definition~\ref{def:dia}. Further, we gave in \cite[Theorem~8.2]{quasi-homogeneity-rev} a sufficient condition on $d$, namely \cite[Assumption~7.13]{quasi-homogeneity-rev}, for $f_{P,d}$ to have a dense open orbit in $\overline{M}_{0,3}(X,d)$ under the action of $G$. This prepared well Step~\eqref{item:step:nec_and_suff} of the proof of Theorem~\ref{thm:main2-intro}. In this work, we build on these results to prove the following uniform theorem.

\begin{thm}[Theorem~\ref{thm:main2}]
\label{thm:main2-intro}

Let $d$ be a minimal degree in $H_2(X)$. The morphism $f_{P,d}$ has a dense open orbit in $\overline{M}_{0,3}(X,d)$ under the action of $\operatorname{Aut}(X)$. In particular, the moduli space $\overline{M}_{0,3}(X,d)$ is quasi-homogeneous under the action of $\operatorname{Aut}(X)$.

\end{thm}

\begin{proof}[Idea of a proof]

All ingredients for the proof of Theorem~\ref{thm:main2-intro} are actually already contained in \cite{quasi-homogeneity-rev}. However, it became only apparent to the author after completing the manuscript \cite{quasi-homogeneity-rev} how to use these ingredients in a precise manner. Roughly, we proceed in two steps.
\begin{enumerate}
    \item\label{item:step:nec_and_suff} 
    Let $d$ be a minimal degree in $H_2(X)$. In \cite[Assumption~7.13]{quasi-homogeneity-rev}, we gave a sufficient condition on $d$ for $f_{P,d}$ to have a dense open orbit in $\overline{M}_{0,3}(X,d)$ under the action of $G$ (cf.~\cite[Theorem~8.2]{quasi-homogeneity-rev}). In the first step, we replace, by various refinements and strengthenings of theorems in \cite{quasi-homogeneity-rev} treated in Section~\ref{sec:additional}~to~\ref{sec:proof}, the sufficient condition on $d$, namely \cite[Assumption~7.13]{quasi-homogeneity-rev}, by a necessary and sufficient condition on $d$, namely Assumption~\ref{ass:not_G2/P1}, for $f_{P,d}$ to have a dense open orbit in $\overline{M}_{0,3}(X,d)$ under the action of $G$. The final result of this step is Theorem~\ref{thm:main1}. Assumption~\ref{ass:not_G2/P1} excludes precisely one instance of $G$, $P$ and $d$, namely in type $\mathsf{G}_2$, from our considerations. 
    
    The main idea of the proof of Theorem~\ref{thm:main1} is the construction of additional tangent directions which is carried out in Section~\ref{sec:additional} (see Lemma~\ref{lem:TDtilde} and Definition~\ref{def:tildeTD} for a precise definition). This construction was prepared and is motivated by positivity in generalized cascades of orthogonal roots from \cite[Section~5]{quasi-homogeneity-rev}.
    \item\label{item:step:G2}
    To finish the proof of Theorem~\ref{thm:main2-intro}, it is sufficient by Step~\eqref{item:step:nec_and_suff} to consider the instance of $G$, $P$ and $d$ in type $\mathsf{G}_2$ discarded in Assumption~\ref{ass:not_G2/P1}. We treat this case by passage from $G$ to $\operatorname{Aut}(X)$. The necessary analysis and computations concerning the inclusion $G_2\subseteq B_3$ are performed in Appendix~\ref{app:G2}.\qedhere
\end{enumerate}
\end{proof}

\subsection*{Organization}

Section~\ref{sec:notation}~to~\ref{sec:minimal} are supposed to set up notation and terminology which is used for the rest of the paper. This notation and terminology concerns mostly a recapitulation of the aspects of the theory of minimal degrees developed in \cite{minimal,quasi-homogeneity-rev}. By the nature of the problem, the refinements and strengthenings of theorems in \cite{quasi-homogeneity-rev} treated in Section~\ref{sec:additional}~to~\ref{sec:proof} which lead to the proof of Theorem~\ref{thm:main1} discussed in Step~\eqref{item:step:nec_and_suff} above concern \enquote{special} combinatorics, in the sense that these combinatorics are only non-trivial for minimal degrees in $H_2(X)$ which do not satisfy \cite[Assumption~7.13]{quasi-homogeneity-rev}, in particular only if the root system associated to $G$ and $T$ is not simply laced. A reader only interested in Theorem~\ref{thm:main1} subject to Step~\eqref{item:step:nec_and_suff} does not need to read Appendix~\ref{app:G2}. The appendix on the inclusion $G_2\subseteq B_3$ is independent from the rest of the paper, in the sense that it is only used in the proof of Theorem~\ref{thm:main2} which deals with Step~\eqref{item:step:G2} above. Vice versa, none of the considerations in the main text are needed to follow Appendix~\ref{app:G2}.

\subsection*{Acknowledgment}

The support of the German Research Foundation (DFG) is gratefully acknowledged. The author wants to thank all users of \TeX\ StackExchange and MathOverflow who answered his questions. In particular, \cite{victor,venka} leaded us to the right references \cite{demazure,m.postnikov} which helped to complete the work on this paper.

\section{Notation}
\label{sec:notation}

In addition to the notation from the introduction, we fix once and for all further notation related to the theory of linear algebraic groups. For general background and more details concerning this theory, we refer to \cite{bourbaki_roots,humphreys_groups,Humphreys_Lie}.
\addtocounter{footnote}{1}
\footnotetext{Here and in what follows, we use the suggestive notation $\mathbb{R}\Delta$ for the $\mathbb{R}$-span of $\Delta$. We will use similar notation for other $\mathbb{R}$-spans and $\mathbb{Z}$-spans.}
\addtocounter{footnote}{-1}
\allowdisplaybreaks\begin{align*}
    &W&&
    \mathrlap{\text{the Weyl group associated to $G$ and $T$,}}\hphantom{\text{the Levi factor of $\tilde{P}_1$ with Lie algebra $\tilde{\mathfrak{l}}_1$ such that $\tilde{R}_{\tilde{P}_1}$}}\\
    &W_P&&
    \text{the Weyl group associated to $P$ and $T$,}\\
    &R&&
    \text{the root system associated to $G$ and $T$,}\\
    &R_P&&
    \begin{tabular}{@{}l@{}}
    the root system associated to the Levi factor of $P$\\
    and $T$,
    \end{tabular}\\
    &\Delta&&
    \text{the base of $R$ corresponding to $B$,}\\
    &\Delta_P&&
    \begin{tabular}{@{}l@{}}
    the set of simple roots corresponding to $P$ forming a\\
    base of $R_P$ (cf.~\cite[Theorem~29.3]{humphreys_groups}),
    \end{tabular}\\
    &R^+&&
    \text{the set of positive roots of $R$ with respect to $\Delta$,}\\
    &R_P^+&&
    \text{the set of positive roots of $R_P$ with respect to $\Delta_P$,}\\
    &R^-&&
    \text{the set of negative roots of $R$ with respect to $\Delta$,}\\
    &R_P^-&&
    \text{the set of negative roots of $R_P$ with respect to $\Delta_P$,}\\  
    &\mathrlap{(-,-)}&&
    \begin{tabular}{@{}l@{}}
    a $W$-invariant scalar product on $\mathbb{R}\Delta$ uniquely deter-\\
    mined up to scalar in $\mathbb{R}_{>0}$,\footnotemark
    \end{tabular}\\
    &{\leq}&&
    \begin{tabular}{@{}l@{}}
    the partial order on $R$ with set of positive elements\\
    given by $R^+$,
    \end{tabular}\\       
    &\alpha^\vee&&    
    \text{the coroot/dual of a root $\alpha\in R$ defined as $\tfrac{2\alpha}{(\alpha,\alpha)}$,}\\
    &\omega_\beta&&
    \begin{tabular}{@{}l@{}}
    the fundamental weight associated to $\beta\in\Delta$ defined\\
    by the equation $(\omega_\beta,\beta'^\vee)=\delta_{\beta,\beta'}$ for all $\beta,\beta'\in\Delta$, 
    \end{tabular}\\  
    &s_\alpha&&
    \begin{tabular}{@{}l@{}}
    the reflection along the hyperplane perpendicular to\\
    $\alpha$ defined as $\mathbb{R}\Delta\to\mathbb{R}\Delta\,,\,\lambda\mapsto\lambda-(\lambda,\alpha^\vee)\alpha$,
    \end{tabular}\\
    &\Delta^\vee&&
    \text{the set of simple coroots of simple roots in $\Delta$,}\\
    &\Delta_P^\vee&&
    \text{the set of simple coroots of simple roots in $\Delta_P$,}\\
    &\mathrlap{H_*(X)}&&
    \text{the homology ring with integral coefficients,}\\
    &\mathrlap{H^*(X)}&&
    \text{the cohomology ring with integral coefficients,}\\
    &\mathrlap{H_2(X)}&&
    \text{will be identified with $\mathbb{Z}\Delta^\vee/\mathbb{Z}\Delta_P^\vee$ as in \cite[p.~260]{curvenbhd},}\\
    &\mathrlap{H^2(X)}&&
    \begin{tabular}{@{}l@{}}
    will be identified with $\mathbb{Z}\{\omega_\beta\mid\beta\in\Delta\setminus\Delta_P\}$ as in\\
    \cite[p.~260]{curvenbhd},
    \end{tabular}\\
    &\mathrlap{(-,-)}&&
    \begin{tabular}{@{}l@{}}
    the Poincar\'e pairing $H^2(X)\otimes H_2(X)\to\mathbb{Z}$ given by\\
    $(\omega_\beta,\beta'^\vee+\mathbb{Z}\Delta_P^\vee)=\delta_{\beta,\beta'}$ for all $\beta,\beta'\in\Delta\setminus\Delta_P$,
    \end{tabular}\\
    &{\leq}&&
    \begin{tabular}{@{}l@{}}
    the partial order on $H_2(X)$ with set of positive ele-\\
    ments given by $d\in H_2(X)$ such that $(\omega_\beta,d)\geq 0$ for\\
    all $\beta\in\Delta\setminus\Delta_P$ -- this coincides with the partial order\\
    defined in Definition~\ref{def:minimal1} and gives in particular a par-\\
    ial order on $H_2(G/B)=\mathbb{Z}\Delta^\vee$,\footnotemark
    \end{tabular}\\
    &\mathrlap{c_1(X)}&&
    \begin{tabular}{@{}l@{}}    
    the first Chern class of the tangent bundle on $X$\\
    identified with $\sum_{\alpha\in R^+\setminus R_P^+}\alpha\in H^2(X)$ as in\\
    \cite[Lemma~3.5]{fulton-woodward},
    \end{tabular}\\
    &\ell&&
    \text{the length function $W\to\mathbb{Z}_{\geq 0}$ on $W$,}\\
    &{\leq}&&
    \text{the Bruhat order on $W$,}\\
    &w_o&&
    \text{the longest element of $W$,}\\
    &w_P&&
    \text{the longest element of $W_P$,}\\
    &\mathrlap{I(w)}&&
    \begin{tabular}{@{}l@{}}
    the inversion set of an element $w\in W$ given by the\\
    set of all $\alpha\in R^+$ such that $w(\alpha)<0$,
    \end{tabular}
\end{align*}\allowdisplaybreaks[0]
\footnotetext{In a similar fashion, we also have an analogously defined partial order \enquote{$\leq$} on $\mathbb{Z}\Delta/\mathbb{Z}\Delta_P$. This gives in particular a partial order on $\mathbb{Z}\Delta$ which restricts to the previously defined partial order on $R$. This partial order will be in use once in the proof of Lemma~\ref{lem:TDtilde}\eqref{item:in_tangent}.}
\begin{rem}

In the above list of notation, we gave multiple meanings to the pairing $(-,-)$ and to the partial order \enquote{$\leq$}. However, each of the distinct meanings is defined on distinct mathematical objects, so that no confusion arises.

\end{rem}

\begin{rem}

In general, for a subset $S$ of $R$, we denote by $-S$ the subset of $R$ given by $\{-\alpha\mid\alpha\in S\}$. With this notation, we have for example $R^-=-R^+$ and $R_P^-=-R_P^+$.

\end{rem}

\section{Minimal degrees}
\label{sec:minimal}

In this section, we recall the aspects of the theory of minimal degrees which will be needed in the course of this work. These aspects were developed in \cite{minimal,quasi-homogeneity-rev} based on \cite{curvenbhd,fulton-woodward,kostant,postnikov}. For more details, we refer to these papers. A preliminary definition of minimal degrees was already given in Definition~\ref{def:minimal1}. However, for most purposes, an equivalent combinatorial definition of minimal degrees in terms of curve neighborhoods, namely Definition~\ref{def:minimal2}, is more suitable. The notation and terminology in this section will be fixed once and for all for the rest of the paper.

\begin{defn}[{\cite[Section~4.2]{curvenbhd}}]

Let $d$ be a degree in $H_2(X)$. The maximal elements of the set $\{\alpha\in R^+\setminus R_P^+\mid \alpha^\vee+\mathbb{Z}\Delta_P^\vee\leq d\}$ are called maximal roots of $d$. A sequence of roots $(\alpha_1,\ldots,\alpha_r)$ is called a greedy decomposition of $d$ if $\alpha_1$ is a maximal root of $d$ and $(\alpha_2,\ldots,\alpha_r)$ is a greedy decomposition of $d-(\alpha_1^\vee+\mathbb{Z}\Delta_P^\vee)$. The empty sequence is the unique greedy decomposition of $0$.

\end{defn}

\begin{rem}

Let $d$ be a degree in $H_2(X)$. According to \cite[Section~4]{curvenbhd}, the greedy decomposition of $d$ is unique up to reordering.

\end{rem}

\begin{defn}[{\cite[Section~4.2]{curvenbhd}}]

A root $\alpha\in R^+\setminus R_P^+$ is called $P$-cosmall if $\alpha$ is a maximal root of $\alpha^\vee+\mathbb{Z}\Delta_P^\vee$. In particular, we can speak about $B$-cosmall roots.

\end{defn}

\begin{defn}[{\cite[Section~4.2]{curvenbhd}}]

Let $d$ be a degree in $H_2(X)$. Let $(\alpha_1,\ldots,\alpha_r)$ be a greedy decomposition of $d$. Then we define an element $z_d^P\in W$ by the following equation
\[
z_d^Pw_P=s_{\alpha_1}\cdot\ldots\cdot s_{\alpha_r}\cdot w_P\,,
\]
where $\cdot$ denotes the Hecke product in $W$ (cf.~\cite[Section~3]{curvenbhd} for a definition and basic properties of this monoid structure on $W$). It is easy to see that $z_d^P$ is the minimal representative in $z_d^PW_P$ (cf. \cite[Proposition~2.4(8)]{minimal}). Well-definedness questions of the element $z_d^P$ (independence of the choice of the greedy decomposition of $d$) are discussed in detail in \cite[Section~4, in particular Definition~4.6]{curvenbhd}.

\end{defn}

\begin{rem}

The geometric meaning of the element $z_d^P\in W$ for a degree $d\in H_2(X)$ is illuminated by the theory of curve neighborhoods, cf.~\cite[Theorem~5.1]{curvenbhd}.

\end{rem}

\begin{defn}
\label{def:minimal2}

Let $d$ be a degree in $H_2(X)$. We say that $d$ is a minimal degree in $H_2(X)$ if $d$ is a minimal element of the set
\[
\left\{d'\text{ a degree in $H_2(X)$ such that }z_d^P\leq z_{d'}^P\right\}
\]
with respect to the partial order \enquote{$\leq$} on $H_2(X)$.

\end{defn}

\begin{notation}

We denote by $\Pi_P$ the set of all minimal degrees in $H_2(X)$. In particular, the set of all minimal degrees in $H_2(G/B)$ is denoted by $\Pi_B$.

\end{notation}

\begin{rem}

Based on the results in \cite{minimal,curvenbhd,fulton-woodward}, it was remarked in \cite[Remark~3.27]{quasi-homogeneity-rev} that Definition~\ref{def:minimal1} and Definition~\ref{def:minimal2} are equivalent.

\end{rem}

\begin{rem}
\label{rem:unique}

Let $d\in\Pi_P$. Then, $d$ is the \emph{unique} minimal element of the set
\[
\left\{d'\text{ a degree in $H_2(X)$ such that }z_d^P\leq z_{d'}^P\right\}\,.
\]
This follows from \cite[Corollary~3]{postnikov} by the explanations based on \cite{minimal,curvenbhd,fulton-woodward} given in \cite[Remark~3.27]{quasi-homogeneity-rev}.

\end{rem}

\begin{defn}

By Remark~\ref{rem:unique}, there exists a unique minimal degree $d_X\in\Pi_P$ such that $d_X$ is the unique minimal element of the set
\[
\left\{d\text{ a degree in $H_2(X)$ such that }w_oW_P=z_d^PW_P\right\}\,.
\]
We denote this unique minimal degree in $H_2(X)$ from now on always by $d_X$. In particular, we can speak about $d_{G/B}\in\Pi_B$.

\end{defn}

\begin{rem}
\label{rem:dX_point}

By \cite[Corollary~3]{postnikov} and the explanations given in \cite[Remark~3.27]{quasi-homogeneity-rev} based on \cite{minimal,curvenbhd,fulton-woodward}, we know that $d_X$ is the unique minimal degree in the quantum product of two point classes in $H^*(X)$ in the sense of Definition~\ref{def:minimal1}.

\end{rem}

\begin{defn}

Let $e\in\Pi_B$. Then we define 
$$
\mathcal{B}_{R,e}=\left\{\alpha\in R^+\;\middle|\;\alpha\text{ occurs in a greedy decomposition of }e\right\}\,.
$$
We call the set $\mathcal{B}_{R,e}$ of positive roots a generalized cascade of orthogonal roots. This name is justified because two distinct elements of $\mathcal{B}_{R,e}$ are strongly orthogonal in the sense of Definition~\ref{def:strongly}\eqref{item:strongly} by \cite[Theorem~4.5(3)]{quasi-homogeneity-rev}.

\end{defn}

\begin{notation}
\label{notation:cascade}

We set $\mathcal{B}_R=\mathcal{B}_{R,d_{G/B}}$ for short. By the explanations given in \cite[Remark~4.2]{quasi-homogeneity-rev} based on \cite{minimal}, we know that $\mathcal{B}_R$ is a ordinary cascade of orthogonal roots in the sense of \cite[Section~1]{kostant}. The generalized cascade of orthogonal roots $\mathcal{B}_{R,e}$ for some $e\in\Pi_B$ is therefore a generalization of the ordinary cascade of orthogonal roots $\mathcal{B}_R$. This is another justification for the previous definition.

\end{notation}

\begin{defn}

Let $d\in\Pi_P$. The lifting $e$ of $d$ is the unique minimal degree $e\in\Pi_B$ such that $z_d^Pw_P=z_e^B$. Uniqueness and existence of the lifting of $d$ were discussed in \cite[Definition~6.2, Fact~6.5(1)]{quasi-homogeneity-rev} based on \cite{minimal,postnikov}.

\end{defn}

\begin{notation}[{\cite[Lemma~4.2]{fulton-woodward}}]

Let $\alpha\in R^+\setminus R_P^+$. We denote by $C_\alpha\subseteq X$ the unique irreducible $T$-invariant curve containing the $T$-fixed points $1P$ and $s_\alpha P$. By \cite[Lemma~4.2]{fulton-woodward} such a unique curve exists. Moreover, \cite[Lemma~4.2]{fulton-woodward} says that $C_\alpha$ is isomorphic to $\mathbb{P}_{\mathbb{C}}^1$. For an explicit construction of $C_\alpha$, see \cite[Section~3]{fulton-woodward}.

\end{notation}

\begin{defn}
\label{def:dia}

Let $d\in\Pi_P$. Let $e$ be the lifting of $d$. Then we define a morphism $f_{P,d}$ by the assignment 
$$
f_{P,d}\colon\mathbb{P}_{\mathbb{C}}^1\hookrightarrow\prod_{\alpha\in\mathcal{B}_{R,e}\setminus R_P^+}C_\alpha\hookrightarrow X
$$
where the first morphism is the diagonal embedding of $\mathbb{P}_{\mathbb{C}}^1$ into $\lvert\mathcal{B}_{R,e}\setminus R_P^+\rvert$ isomorphic copies of $\mathbb{P}_{\mathbb{C}}^1$ and the second morphism is the embedding into $X$ which is well-defined due to \cite[Theorem~4.5(3)]{quasi-homogeneity-rev}. Again by \cite[Theorem~4.5(3)]{quasi-homogeneity-rev}, the definition of $f_{P,d}$ is independent of the ordering of the product $\smash{\prod_{\alpha\in\mathcal{B}_{R,e}\setminus R_P^+}}$. Hence, the morphism $f_{P,d}$ is well-defined.

\end{defn}

\begin{rem}
\label{rem:in_M(2)}

Let $d\in\Pi_P$. By \cite[Fact~7.4]{quasi-homogeneity-rev}, we know that $f_{P,d}$ satisfies $(f_{P,d})_*[\mathbb{P}_{\mathbb{C}}^1]=d$, $f_{P,d}(0)=1P$ and $f_{P,d}(\infty)=z_d^P P$. Hence, $f_{P,d}$ is an element of $\overline{M}_{0,3}(X,d)$ and even an element of $\overline{M}_{0,3}(X,d)(2)$ where the last moduli space is defined in Notation~\ref{not:M(2)}.

\end{rem}

\section{Maximal sets of pairwise strongly orthogonal roots}

In this section, we draw easy consequences of the classification of maximal sets of pairwise strongly orthogonal roots \cite{MSOS} and formulate them in the language of generalized cascades of orthogonal roots. All statements in this section can be readily deduced from \cite{MSOS}. In the end, for the proof of Lemma~\ref{lem:ass}, we will only need Corollary~\ref{cor:less_than_d_X} and Lemma~\ref{lem:equivalent_to_cascade} restricted to a root system of type $\mathsf{G}_2$. However, we present the results in a systematic generality to be clear.

\begin{defn}[{\cite[Definition~1.5,~2.1,~2.2]{MSOS}}]
\label{def:strongly}

\leavevmode
\begin{enumerate}
    \item\label{item:strongly}
    Two roots $\alpha$ and $\gamma$ are called strongly orthogonal if $\alpha\pm\gamma\notin R\cup\{0\}$ holds.
    \item
    A subset $\Gamma$ of $R$ is called a set of pairwise strongly orthogonal roots ($={}$SOS) if $\alpha$ and $\gamma$ are strongly orthogonal for all distinct elements $\alpha,\gamma\in\Gamma$.
    \item
    A subset $\Gamma$ of $R$ is called a maximal set of pairwise strongly orthogonal roots ($={}$MSOS) if $\Gamma$ is a SOS and if $\Gamma$ is not a proper subset of any other SOS.
    \item
    A subset $\Gamma$ of $R$ is called a set of pairwise strongly orthogonal roots of maximal cardinality ($={}$MMSOS) if $\Gamma$ is a SOS and if the cardinality of any other SOS is less or equal than the cardinality of $\Gamma$. Clearly, every MMSOS is an MSOS.
    \item
    Two SOSs $\Gamma$ and $\Gamma'$ are said to be equivalent, in formulas $\Gamma\sim\Gamma'$, if there exists an element $w\in W$ such that $\Gamma'=w\Gamma$. Clearly, \enquote{$\sim$} is an equivalence relation on the set of all SOSs which preserves cardinality and inclusion. Hence, \enquote{$\sim$} also preserves MSOSs and MMSOSs.
\end{enumerate}

\end{defn}

\begin{lem}[{\cite[Theorem~3.1,~5.1]{MSOS}}]
\label{lem:MMSOS}

Every MMSOS is equivalent to $\mathcal{B}_R$. In other words, there exists a unique equivalence class of MMSOSs, and $\mathcal{B}_R$ is a representative of it.  

\end{lem}

\begin{proof_header}[Proof of Lemma~\ref{lem:MMSOS} if there exists a unique equivalence class of MSOSs]
Assume first that there exists a unique equivalence class of MSOSs. Then, there also exists a unique equivalence class of MMSOSs which is equal to the former. To see that $\mathcal{B}_R$ is a representative of this class, it suffices to prove that $\mathcal{B}_R$ is an MMSOS or equivalently an MSOS. In~\cite[Theorem~1.8]{kostant}, it was proved in general that $\mathcal{B}_R$ is an MSOS. Hence, we are done.
\end{proof_header}

\begin{proof_header}[Proof of Lemma~\ref{lem:MMSOS} if there exist several equivalence classes of MSOSs]
Assume next that there exist several equivalence classes of MSOSs. It follows from the classification of MSOSs in \cite[Theorem~3.1 (for classical types), Theorem~5.1 (for exceptional types)]{MSOS} that $R$ is of one of the types in Table~\ref{table:MSOS}
\begin{table}[ht]
    \centering
    {\renewcommand{\arraystretch}{1.1}
    \begin{tabular}{lc}
        Type of $R$ & $\left|\mathcal{B}_R\right|$ \\
        \hline\hline
        $\mathsf{B}_l$, $l\in\mspace{1.5mu}2\mathbb{Z}_{>0}$ & $l$ \\
        $\mathsf{C}_l$, $l\geq 2$ & $l$ \\
        $\mathsf{F}_4$ & $4$ 
    \end{tabular}}
    \caption{Cardinalities of $\mathcal{B}_R$\\
    for different types of $R$.}
    \label{table:MSOS}
\end{table}
and that there exists a unique equivalence class of MMSOSs. In view of \cite[Theorem~1.8]{kostant}, it suffices to prove that the cardinality of a MMSOS equals the cardinality of $\mathcal{B}_R$. This is accomplished by comparison of \cite[Table~3,~6]{MSOS} with Table~\ref{table:MSOS}.
\end{proof_header}

\begin{cor}
\label{cor:less_than_d_X}

Let $e\in\Pi_B$. Then, we have $\left|\mathcal{B}_{R,e}\right|\leq\left|\mathcal{B}_R\right|$.

\end{cor}

\begin{proof}

Let $e\in\Pi_B$. By \cite[Theorem~4.5(3)]{quasi-homogeneity-rev}, we know that $\mathcal{B}_{R,e}$ is an SOS. The result therefore follows from Lemma~\ref{lem:MMSOS} and the definition of MMSOSs.
\end{proof}

\begin{notation}

We denote the center of $W$ by $Z(W)$ in what follows.

\end{notation}

\begin{lem}
\label{lem:center}

We have $Z(W)\subseteq\{1,w_o\}$ with equality if and only if $w_o=-1$.

\end{lem}

\begin{proof}

Let $w\in Z(W)$. For all $\beta\in\Delta$, we have $s_\beta=ws_\beta w^{-1}=s_{w(\beta)}$ because $w$ is central, and thus $w(\beta)=\pm\beta$. On the other hand, since $w(\theta_1)\in R$ where $\theta_1$ is the highest root of $R$ and since every simple root occurs in the support of $\theta_1$, we conclude that the sign in the expression $w(\beta)=\pm\beta$ is independent of the choice of $\beta\in\Delta$. We see that either $w(\beta)=\beta$ for all $\beta\in\Delta$, in which case $w=1$, or that $w(\beta)=-\beta$ for all $\beta\in\Delta$, in which case $w=w_o=-1$. The result follows from this.
\end{proof}

\begin{ex}
\label{ex:center}

If $R$ is non simply laced, then $w_o=-1$ as it follows from inspection of \cite[Plate~II,~III,~VIII,~IX]{bourbaki_roots}. If $R$ is non simply laced, we therefore have by Lemma~\ref{lem:center} that $Z(W)=\{1,w_o\}$, and in particular that $w_o\in Z(W)$. However, the non simply laced root systems do not cover all cases where $w_o=-1$ as the rest of the plates in loc.~cit.\ shows.

\end{ex}

\begin{lem}
\label{lem:equivalent_to_cascade}

Suppose that $w_o\in Z(W)$. Let $e\in\Pi_B$ be such that $\left|\mathcal{B}_{R,e}\right|=\left|\mathcal{B}_R\right|$. Then, we have $e=d_{G/B}$.

\end{lem}

\begin{proof}

Let $e$ be as in the statement. By \cite[Theorem~4.5(3)]{quasi-homogeneity-rev}, we know that $\mathcal{B}_{R,e}$ is an SOS, and by assumption and Lemma~\ref{lem:MMSOS}, it follows that $\mathcal{B}_{R,e}$ is even an MMSOS. By Lemma~\ref{lem:MMSOS} again, we see that $\mathcal{B}_R\sim\mathcal{B}_{R,e}$. Thus, there exists a $w\in W$ such that $\mathcal{B}_{R,e}=w\mathcal{B}_R$. By \cite[Theorem~4.5(3), Theorem~4.7]{quasi-homogeneity-rev}, this means that $z_e^B=ww_ow^{-1}$. Because $w_o$ is central by assumption, we conclude that $z_e^B=w_o$. \cite[Corollary~6.7]{quasi-homogeneity-rev} based on \cite{postnikov} now implies that $e=d_{G/B}$ -- as desired.
\end{proof}

\section{Additional tangent directions}
\label{sec:additional}

In this section, we construct additional tangent directions using positivity in generalized cascades of orthogonal roots from \cite[Section~5]{quasi-homogeneity-rev}. Different additional tangent directions associated to a minimal degree $d$ in $H_2(X)$ will give rise to linearly independent tangent vectors in the tangent space $T_{f_{P,d}}$ associated to the morphism $f_{P,d}$ defined in Definition~\ref{def:dia} (cf.~Notation~\ref{notation:TfPd}, Lemma~\ref{lem:prel_main1}).

\begin{defn}

Let $d\in\Pi_P$. Let $e$ be the lifting of $d$. We define the following set
\[
\mathrm{TD}_{P,d}=\left\{-\alpha-\gamma\in R^-\setminus R_P^-\;\middle|\;\alpha\in\mathcal{B}_{R,e}\setminus R_P^+,\gamma\in R_P^+\cup\{0\}\right\}
\]
which we call the set of tangent directions (associated to $d$). Further, we call an element of this set a tangent direction (associated to $d$).

\end{defn}

\begin{lem}
\label{lem:TDtilde}

Let $d\in\Pi_P$. Let $e$ be the lifting of $d$. Let $\alpha\in\mathcal{B}_{R,e}\setminus R_P^+$ and $\gamma\in R_P^+$ be such that $(\alpha,\gamma)<0$. Let $\alpha'\in\mathcal{B}_{R,e}\setminus R_P^+$ be the unique root such that $(\alpha',\gamma)>0$ (cf.~\cite[Lemma~5.3, Theorem~5.5]{quasi-homogeneity-rev}). Let $\gamma'=-z_e^B(\gamma)$ for short. The following items hold.
\begin{enumerate}
    \item\label{item:in_L}
    We have $\gamma'\in R^+$ and $z_e^B(\gamma')\in R_P^-$.
    \item\label{item:prep_in_tangent}
    We have $(\gamma,\alpha'^\vee)=1$.
    \item\label{item:in_tangent}
    We have $\alpha'-\gamma'\in R^+\setminus R_P^+$.
\end{enumerate}
If in addition $(\gamma,\alpha^\vee)<-1$, then the following items also hold.
\begin{enumerate}[resume]
    \item\label{item:long_short}
    The root $\alpha$ is short and the root $\gamma$ is long.
    \item\label{item:dual_of_prep_in_tangent}
    We have $(\alpha',\gamma^\vee)=1$.
    \item\label{item:not_in_TD}
    We have $-\alpha'+\gamma'\notin\mathrm{TD}_{P,d}$.
\end{enumerate}

\end{lem}

\begin{proof}[Proof of Item~\eqref{item:in_L}]

Let the notation be as in the statement. It is clear that $z_e^B(\gamma')=-\gamma\in R_P^-$ by assumption. Note that $z_e^B$ is an involution by \cite[Corollary~4.9]{curvenbhd}. By \cite[Fact~6.5(2)]{quasi-homogeneity-rev}, the element $z_e^B$ is a maximal representative in $z_e^B W_P$. Hence, $z_e^B(\gamma)<0$ because $\gamma\in R_P^+$.
By definition, it follows that $\gamma'\in R^+$.
\end{proof}

\begin{proof}[Proof of Item~\eqref{item:prep_in_tangent}]

Since $\alpha'\in\mathcal{B}_{R,e}\setminus R_P^+$, $\gamma\in R_P^+$, $(\alpha',\gamma)>0$, it is clear that $\gamma\in I(s_{\alpha'})\setminus\{\alpha'\}$. The equation $(\gamma,\alpha'^\vee)=1$ then follows from \cite[Theorem~4.5(2)]{quasi-homogeneity-rev} and \cite[Theorem~6.1(c)]{curvenbhd}. 
\end{proof}

\begin{proof}[Proof of Item~\eqref{item:in_tangent}]

By Item~\eqref{item:prep_in_tangent}, we know that $s_{\alpha'}(\gamma)=-\alpha'+\gamma$ is a root. Thus, $z_e^B(-\alpha'+\gamma)=-z_e^B(\alpha')-\gamma'$ is also a root. By \cite[Theorem~4.5(3), Theorem~4.7]{quasi-homogeneity-rev}, we have $z_e^B(\alpha')=-\alpha'$. Altogether, it follows that $\alpha'-\gamma'$ is a root. In view of \cite[loc.~cit.]{quasi-homogeneity-rev} and Item~\eqref{item:prep_in_tangent}, we compute
\[
\alpha'-\gamma'+\mathbb{Z}\Delta_P=\alpha'+z_e^B(\gamma)+\mathbb{Z}\Delta_P=-\sum_{\bar{\alpha}\in\mathcal{B}_{R,e}\setminus(R_P^+\cup\{\alpha'\})}(\gamma,\bar{\alpha}^\vee)\bar{\alpha}+\mathbb{Z}\Delta_P\,.
\]
By~\cite[Lemma~5.3]{quasi-homogeneity-rev} and assumption, the above expression is $\geq\alpha+\mathbb{Z}\Delta_P$. Since $\alpha\in R^+\setminus R_P^+$, we see that also $\alpha'-\gamma'\in R^+\setminus R_P^+$.
\end{proof}

\begin{proof}[Proof of Item~\eqref{item:long_short}]

Assume from now on that $(\gamma,\alpha^\vee)<-1$. This assumption implies directly that $\alpha$ is short and that $\gamma$ is long, since $\alpha$ and $\gamma$ are clearly non-proportional (cf.~\cite[Lemma~7.14]{quasi-homogeneity-rev}).
\end{proof}

\begin{proof}[Proof of Item~\eqref{item:dual_of_prep_in_tangent}]

We know that $(\alpha',\gamma)>0$, that $\alpha'$ and $\gamma$ are non-proportional, and that $\gamma$ is long by the previous item. The claimed equality therefore follows from \cite[loc.~cit.]{quasi-homogeneity-rev}.
\end{proof}

\begin{proof}[Proof of Item~\eqref{item:not_in_TD}]

By Item~\eqref{item:prep_in_tangent}, we see that $(\alpha'-\gamma',\alpha'^\vee)=1$. This implies that $\alpha'-\gamma'\notin\mathcal{B}_{R,e}$ by \cite[Theorem~4.5(3)]{quasi-homogeneity-rev}. If we assume now for a contradiction that $-\alpha'+\gamma'\in\mathrm{TD}_{P,d}$, we find $\bar{\alpha}\in\mathcal{B}_{R,e}\setminus R_P^+$ and $\bar{\gamma}\in R_P^+$ such that $\alpha'-\gamma'=\bar{\alpha}+\bar{\gamma}$. By assumption, we see that
\[
(\bar{\alpha}+\bar{\gamma},\alpha^\vee)=(\alpha'-\gamma',\alpha^\vee)=-(\gamma',\alpha^\vee)=-(\gamma,\alpha^\vee)>1
\]
where we used in the last equality the $W$-invariance of the scalar product applied to $z_e^B$ and similar arguments as in the proof of Item~\eqref{item:in_L},~\eqref{item:in_tangent}. Since $\bar{\alpha}+\bar{\gamma}\neq\alpha$ by the first reasoning in the proof of this item, we infer from \cite[Theorem~4.5(2)]{quasi-homogeneity-rev}, \cite[Theorem~6.1(c)]{curvenbhd} and the last displayed inequality that $\bar{\alpha}+\bar{\gamma}\notin I(s_\alpha)$. This means that
\[
s_\alpha(\bar{\alpha}+\bar{\gamma})=\bar{\alpha}+\bar{\gamma}+(\gamma,\alpha^\vee)\alpha>0\,.
\]
From the last displayed inequality, we conclude that $\alpha\neq\bar{\alpha}$. Indeed, since otherwise, we find by assumption that
$
\bar{\gamma}>(-(\gamma,\alpha^\vee)-1)\alpha\geq\alpha
$
which contradicts the fact that $\bar{\gamma}\in R_P^+$. Since we now know that $\alpha\neq\bar{\alpha}$, we conclude from the first displayed in equation in the proof of this item that $(\bar{\gamma},\alpha^\vee)=-(\gamma,\alpha^\vee)>1$. Since $\alpha\in\mathcal{B}_{R,e}\setminus R_P^+$ and $\bar{\gamma}\in R_P^+$, this implies that $\bar{\gamma}\in I(s_\alpha)\setminus\{\alpha\}$. The inequality $(\bar{\gamma},\alpha^\vee)>1$ eventually contradicts \cite[loc.~cit.]{quasi-homogeneity-rev,curvenbhd}.
\end{proof}

\begin{defn}
\label{def:tildeTD}

Let $d\in\Pi_P$. Let $e$ be the lifting of $d$. Let $\alpha\in\mathcal{B}_{R,e}\setminus R_P^+$ and $\gamma\in R_P^+$ be such that $(\alpha,\gamma)<0$. We say that $(\alpha',\gamma')$ are associated to $(\alpha,\gamma)$ when $(\alpha',\gamma')$ are defined depending on $(\alpha,\gamma)$ as in the statement of Lemma~\ref{lem:TDtilde}. In view of Lemma~\ref{lem:TDtilde}\eqref{item:in_tangent}, we can define the following set
\[
\widetilde{\mathrm{TD}}_{P,d}=\left\{-\alpha'+\gamma'\in R^-\setminus R_P^-\;\middle|\;
\begin{tabular}{@{}c@{}}
$(\alpha',\gamma')$ are associated to $(\alpha,\gamma)$\\
where $\alpha\in\mathcal{B}_{R,e}\setminus R_P^+$ and $\gamma\in R_P^+$\\
are such that $(\gamma,\alpha^\vee)<-1$
\end{tabular}
\right\}
\]
which we call the set of additional tangent directions (associated to $d$). Further, we call an element of this set an additional tangent direction (associated to $d$). By Lemma~\ref{lem:TDtilde}\eqref{item:not_in_TD}, we have a disjoint union
\[
\mathrm{TD}_{P,d}\sqcup\widetilde{\mathrm{TD}}_{P,d}\text{ considered as a subset of $R^-\setminus R_P^-$.}
\]

\end{defn}

\begin{lem}
\label{lem:orth_roots}

Let $\gamma,\alpha_1,\alpha_2\in R$ such that $(\gamma,\alpha_1^\vee)<-1$ and $(\gamma,\alpha_2^\vee)<-1$. Then, $(\alpha_1,\alpha_2)>0$.

\end{lem}

\begin{proof}

In this proof, we use repeatedly that root strings are unbroken (cf.~\cite[9.4]{Humphreys_Lie}). Suppose for a contradiction that $(\alpha_1,\alpha_2)\leq 0$. If we apply $s_{\alpha_1}$ to $\gamma$, we see that $\gamma+\alpha_1$ and $\gamma+2\alpha_1$ are roots. If we apply $s_{\alpha_2}$ to the two previous roots under consideration, we see that $\gamma+\alpha_1+\alpha_2$ and $\gamma+2\alpha_1+2\alpha_2$ are roots. Finally, we compute, similarly as in the proof of Lemma~\ref{lem:TDtilde}\eqref{item:long_short},\eqref{item:dual_of_prep_in_tangent}, that $(\gamma+2\alpha_1+2\alpha_2,\gamma^\vee)=-2$. We conclude that $2\gamma+2\alpha_1+2\alpha_2$ is also a root by applying $s_\gamma$ to the root $\gamma+2\alpha_1+2\alpha_2$. All in all, we see that $\gamma+\alpha_1+\alpha_2$ and $2\gamma+2\alpha_1+2\alpha_2$ are roots. This contradicts the reducedness of the root system.
\end{proof}

\begin{lem}
\label{lem:almost_injective}

Let $d\in\Pi_P$. Let $e$ be the lifting of $d$. Consider the map
\[
\{(\alpha,\gamma)\in(\mathcal{B}_{R,e}\setminus R_P^+)\times R_P^+\mid(\gamma,\alpha^\vee)<0\}\to R^-\setminus R_P^-
\]
defined by the assignment $(\alpha,\gamma)\mapsto-\alpha'+\gamma'$ where $(\alpha',\gamma')$ are associated to $(\alpha,\gamma)$. If two pairs $(\alpha_1,\gamma_1)$ and $(\alpha_2,\gamma_2)$ have the same image, then $(\alpha_1',\gamma_1')=(\alpha_2',\gamma_2')$ where $(\alpha_1',\gamma_1')$ resp.\ $(\alpha_2',\gamma_2')$ are associated to $(\alpha_1,\gamma_1)$ resp.\ $(\alpha_2,\gamma_2)$. In this situation, we also have $\gamma_1=\gamma_2$.

\end{lem}

\begin{proof}

By Lemma~\ref{lem:TDtilde}\eqref{item:in_tangent}, the map in the statement is well-defined. Let $(\alpha_1,\gamma_1)$ and $(\alpha_2,\gamma_2)$ be two pairs which map to the same image $-\alpha_1'+\gamma_1'=-\alpha_2'+\gamma_2'$ where $(\alpha_1',\gamma_1')$ resp.\ $(\alpha_2',\gamma_2')$ are associated to $(\alpha_1,\gamma_1)$ resp.\ $(\alpha_2,\gamma_2)$. By the exact same argument as in the beginning of the proof of Lemma~\ref{lem:TDtilde}\eqref{item:in_tangent}, the last equality is equivalent to $\alpha_1'-\gamma_1=\alpha_2'-\gamma_2$. Assume now for a contradiction that $\gamma_1\neq\gamma_2$. By the last equality, we clearly have $\alpha_1'\neq\alpha_2'$. By \cite[Theorem~4.5(3)]{quasi-homogeneity-rev}, the roots $\alpha_1'$ and $\alpha_2'$ are (strongly) orthogonal.

\begin{proof_header}[Claim: We have $(\gamma_1,\alpha_2'^\vee)=-1$ and $(\gamma_2,\alpha_1'^\vee)=-1$]
Indeed, it suffices to prove the first claimed equality because the situation is symmetric. If we apply $(-,\alpha_2'^\vee)$ to $\alpha_1'-\gamma_1=\alpha_2'-\gamma_2$, we immediately find the desired result in view of Lemma~\ref{lem:TDtilde}\eqref{item:prep_in_tangent} and by what was said directly before the claim.
\end{proof_header}

\noindent
By the previous claim and the analysis before the claim, we know that $(\alpha_1',\gamma_2)$ and $(\alpha_2',\gamma_1)$ are elements of the set
\[
\{(\alpha,\gamma)\in(\mathcal{B}_{R,e}\setminus R_P^+)\times R_P^+\mid(\gamma,\alpha^\vee)=-1\}
\]
which are mapped under the assignment $(\alpha,\gamma)\mapsto-\alpha-\gamma$ to the same element $-\alpha_1'-\gamma_2=-\alpha_2'-\gamma_1$ which, by \cite[Theorem~7.12]{quasi-homogeneity-rev}, necessarily lies in $\mathrm{TD}_{P,d}\setminus(-(\mathcal{B}_{R,e}\setminus R_P^+))$. Again, \cite[loc.~cit.]{quasi-homogeneity-rev} then shows that $(\alpha_1',\gamma_2)=(\alpha_2',\gamma_1)$ and thus $\gamma_1=\gamma_2$. This contradiction shows that we have in the first place that $\gamma_1=\gamma_2$ and equally well $\alpha_1'=\alpha_2'$. Evidently, this also implies that $\gamma_1'=\gamma_2'$.
\end{proof}

\begin{cor}
\label{cor:tildetd_bij}

Let $d\in\Pi_P$. Let $e$ be the lifting of $d$. We have a bijection 
\[
\{(\alpha,\gamma)\in(\mathcal{B}_{R,e}\setminus R_P^+)\times R_P^+\mid(\gamma,\alpha^\vee)<-1\}\xrightarrow{\sim}\widetilde{\mathrm{TD}}_{P,d}
\]
defined by the assignment $(\alpha,\gamma)\mapsto-\alpha'+\gamma'$ where $(\alpha',\gamma')$ are associated to $(\alpha,\gamma)$.

\end{cor}

\begin{proof}

By definition and Lemma~\ref{lem:TDtilde}\eqref{item:in_tangent}, it is clear that the map in the statement of the corollary is well-defined and surjective. Let $(\alpha_1,\gamma_1)$ and $(\alpha_2,\gamma_2)$ be two pairs which map to the same image. By Lemma~\ref{lem:almost_injective}, we know that $\gamma_1=\gamma_2$. Suppose for a contradiction that $\alpha_1\neq\alpha_2$. By \cite[Theorem~4.5(3)]{quasi-homogeneity-rev}, it follows that $\alpha_1$ and $\alpha_2$ are (strongly) orthogonal. Hence, $\alpha_1$ and $\alpha_2$ are two orthogonal roots such that $(\gamma,\alpha_1^\vee)<-1$ and $(\gamma,\alpha_2^\vee)<-1$ where $\gamma=\gamma_1=\gamma_2$. This contradicts Lemma~\ref{lem:orth_roots}. We conclude that $\alpha_1=\alpha_2$. In total, this means that the map is injective, and consequently bijective.
\end{proof}

\section{Refinement of \texorpdfstring{\cite[\textsc{Theorem~7.12}]{quasi-homogeneity-rev}}{[\ref{bib-quasi-homogeneity-rev},~Theorem~7.12]}}

In this section, we provide a refinement of \cite[Theorem~7.12]{quasi-homogeneity-rev}. While writing \cite{quasi-homogeneity-rev}, we have overlooked the fact that the injective map in \cite[Theorem~7.12]{quasi-homogeneity-rev} is defined on a possibly larger set. By relaxing the condition in the domain, we produce in some cases further tangent directions in the image.

\begin{thm}[{Refinement of \cite[Theorem~7.12]{quasi-homogeneity-rev}}]
\label{thm:refinement_7.12}

Let $d\in\Pi_P$. Let $e$ be the lifting of $d$. Then we have an injective map
\[
\{(\alpha,\gamma)\in(\mathcal{B}_{R,e}\setminus R_P^+)\times R_P^+\mid(\alpha,\gamma)<0\}\hookrightarrow\mathrm{TD}_{P,d}\setminus(-(\mathcal{B}_{R,e}\setminus R_P^+))
\]
defined by the assignment $(\alpha,\gamma)\mapsto-\alpha-\gamma$.

\end{thm}

\begin{proof}

We first prove that the map defined by the assignment as in the statement is well-defined. Let $\alpha\in\mathcal{B}_{R,e}\setminus R_P^+$ and $\gamma\in R_P^+$ such that $(\alpha,\gamma)<0$. Since root strings are unbroken  (cf.~\cite[9.4]{Humphreys_Lie}), we know that $\alpha+\gamma$ is a root which, by choice of $\alpha$ and $\gamma$, necessarily lies in $R^+\setminus R_P^+$. Hence, $-\alpha-\gamma\in R^-\setminus R_P^-$ and $-\alpha-\gamma\in\mathrm{TD}_{P,d}$. To finish the proof of well-definedness, assume for a contradiction that $\alpha+\gamma=\bar{\alpha}$ where $\bar{\alpha}\in\mathcal{B}_{R,e}\setminus R_P^+$. By the \cite[second paragraph of the proof of Theorem~7.12]{quasi-homogeneity-rev}, we may assume that $(\gamma,\alpha^\vee)<-1$. Then, $\alpha$ is short, $\gamma$ is long, and $(\alpha,\gamma^\vee)=-1$ by \cite[Lemma~7.14]{quasi-homogeneity-rev}. Therefore, we find that $(\bar{\alpha},\gamma^\vee)=(\alpha+\gamma,\gamma^\vee)=-1+2=1>0$. Thus, $\bar{\alpha}=\alpha'$ where $(\alpha',\gamma')$ are associated to $(\alpha,\gamma)$. By Lemma~\ref{lem:TDtilde}\eqref{item:prep_in_tangent}, we compute $0=(\alpha,\alpha'^\vee)=(\alpha'-\gamma,\alpha'^\vee)=2-1=1$ -- a contradiction. This completes the proof of the well-definedness of the map.

To proof injectivity of the map, suppose that $\alpha_1+\gamma_1=\alpha_2+\gamma_2$ where $\alpha_1,\alpha_2\in\mathcal{B}_{R,e}\setminus R_P^+$ and $\gamma_1,\gamma_2\in R_P^+$ such that $(\alpha_1,\gamma_1)<0$ and $(\alpha_2,\gamma_2)<0$. Suppose for a contradiction that $\alpha_1\neq\alpha_2$. By \cite[Lemma~7.12]{quasi-homogeneity-rev}, we may without loss of generality assume that $(\gamma_2,\alpha_2^\vee)<-1$. As above, this means that $\alpha_2$ is short, $\gamma_2$ is long, and that $(\alpha_2,\gamma_2^\vee)=-1$ by \cite[Lemma~7.14]{quasi-homogeneity-rev}. With this, we compute that 
$
(\alpha_1,\gamma_2^\vee)+(\gamma_1,\gamma_2^\vee)=(\alpha_2+\gamma_2,\gamma_2^\vee)=-1+2=1>0
$.
This means that either $(\alpha_1,\gamma_2)>0$ or $(\gamma_1,\gamma_2)>0$.

\begin{proof_header}[Claim: We have $\alpha_2\neq\alpha_1'$ and thus $(\alpha_2,\gamma_1)\leq 0$ where $(\alpha_1',\gamma_1')$ are associated to $(\alpha_1,\gamma_1)$]
Indeed, suppose for a contradiction that $\alpha_2=\alpha_1'$. With the help of Lemma~\ref{lem:TDtilde}\eqref{item:prep_in_tangent}, we compute that
\[
2+(\gamma_2,\alpha_2^\vee)=(\alpha_1'+\gamma_2,\alpha_1'^\vee)=(\alpha_1+\gamma_1,\alpha_1'^\vee)=0+1=1
\]
and thus $(\gamma_2,\alpha_2^\vee)=-1$. This contradicts the assumption from above that $(\gamma_2,\alpha_2^\vee)<-1$.
\end{proof_header}

\begin{proof}[First case: Suppose that $(\alpha_1,\gamma_2)>0$]

In that case, we know by \cite[loc.~cit.]{Humphreys_Lie} that $\rho=\alpha_1-\gamma_2=\alpha_2-\gamma_1$ is a root which necessarily lies in $R^+\setminus R_P^+$. By exact the same arguments as in the \cite[fifth paragraph of the proof of Theorem~7.12]{quasi-homogeneity-rev}, we conclude that $\alpha_1$ and $\alpha_2$ are comparable, i.e.\ $\alpha_1>\alpha_2$ or $\alpha_1<\alpha_2$. By the same arguments as in the \cite[second claim in the proof of Theorem~7.12]{quasi-homogeneity-rev}, we now have implications
\[
\alpha_1>\alpha_2\Longrightarrow(\alpha_1,\rho)=0\text{ and }
\alpha_1<\alpha_2\Longrightarrow(\alpha_2,\rho)=0\,.
\]
Note that the assumption $(\alpha_1,\gamma_2)>0$ implies that $\alpha_1=\alpha_2'$ where $(\alpha_2',\gamma_2')$ are associated to $(\alpha_2,\gamma_2)$. Lemma~\ref{lem:TDtilde}\eqref{item:prep_in_tangent} shows now that $(\gamma_2,\alpha_1^\vee)=(\gamma_2,\alpha_2'^\vee)=1$. With this, the first implication above yields a contradiction because $(\rho,\alpha_1^\vee)=(\alpha_1-\gamma_2,\alpha_1^\vee)=2-1=1$. The second yields a contradiction because $(\rho,\alpha_2^\vee)=(\alpha_2-\gamma_1,\alpha_2^\vee)\geq 2$ taking into account the previous claim above.
\end{proof}

\begin{proof}[Second case: Suppose that $(\gamma_1,\gamma_2)>0$]

In this case, we know by \cite[loc.~cit.]{Humphreys_Lie} that $\gamma_1-\gamma_2=\alpha_2-\alpha_1$ is a root which necessarily lies in $R_P$. This immediately contradicts the fact that $\alpha_1$ and $\alpha_2$ are strongly orthogonal by \cite[Theorem~4.5(3)]{quasi-homogeneity-rev}.
\end{proof}

\noindent
This finally shows that our initial assumption $\alpha_1\neq\alpha_2$ must be false. We conclude that $\alpha_1=\alpha_2$ and thus $(\alpha_1,\gamma_1)=(\alpha_2,\gamma_2)$ -- as required.
\end{proof}

\section{The key inequality}
\label{sec:key}

The aim of this section is to use the previous results to prove the key inequality, i.e.\ the inequality in Theorem~\ref{thm:key}. This inequality is a refinement of the inequality in \cite[Theorem~7.16]{quasi-homogeneity-rev}, in the sense that we replace \cite[Assumption~7.13]{quasi-homogeneity-rev} by the weaker Assumption~\ref{ass:not_G2/P1} and that we place the set of additional tangent directions on the right side. Assumption~\ref{ass:not_G2/P1} is then (unlike \cite[Assumption~7.13]{quasi-homogeneity-rev}) necessary and sufficient for Theorem~\ref{thm:key} to hold as we see in Example~\ref{ex:fail-key-G2}. The proof of Theorem~\ref{thm:main1} follows easily from the key inequality as we will explain in the next section.

\begin{assumption}
\label{ass:not_G2/P1}

Let $d\in\Pi_P$. We write $(G,P,d)=(G_2,P_1,d_{G_2/P_1})$ if the following assumption holds
\begin{itemize}
    \item 
    $R$ is of type $\mathsf{G}_2$,
    \item
    $P$ is the maximal standard parabolic subgroup with respect to $B$ with set of simple roots $\Delta_P=\{\beta_2\}$ where $\beta_2$ is the simple long root with the labeling of the simple roots as in \cite[Plate~IX]{bourbaki_roots},
    \item
    $d=d_{G/P}$.
\end{itemize}
We further write $(G,P,d)\neq(G_2,P_1,d_{G_2/P_1})$ if $(G,P,d)=(G_2,P_1,d_{G_2/P_1})$ does not hold.

\end{assumption}

\begin{lem}[{Amalgam of \cite[Theorem~6.1(c)]{curvenbhd} and \cite[Theorem~3.4]{quasi-homogeneity-rev}}]
\label{lem:prel_ass}

Let $\alpha$ be a $P$-cosmall root. We have $(\gamma,\alpha^\vee)\in\{0,1\}$ for all $\gamma\in R_P^+$.

\end{lem}

\begin{proof}

Let $\alpha$ be a $P$-cosmall root. By \cite[Theorem~3.4]{quasi-homogeneity-rev}, we know that $(\alpha,\gamma)=0$ for all $\gamma\in R_P^+\setminus I(s_\alpha)$. To finish, for $\gamma\in I(s_\alpha)\cap R_P^+\subseteq I(s_\alpha)\setminus\{\alpha\}$, we have $(\gamma,\alpha^\vee)=1$ by \cite[Theorem~6.1(c)]{curvenbhd}. For this last step, note that $\alpha$ is $P$-cosmall, and hence $B$-cosmall.
\end{proof}

\begin{lem}
\label{lem:prel_ass'}

Let $d\in\Pi_P\setminus\{0\}$. Let $e$ be the lifting of $d$. Assume that $\mathcal{B}_{R,e}$ consists of a unique element $\alpha$. Then, $\alpha$ is a $P$-cosmall root. 

\end{lem}

\begin{proof}

Let the notation be as in the statement. By definition and \cite[Fact~6.5(3)]{quasi-homogeneity-rev}, we know that $d=\alpha^\vee+\mathbb{Z}\Delta_P^\vee\neq 0$. It follows that $\alpha\in R^+\setminus R_P^+$. Let $\bar{\alpha}$ be a maximal root of $d$ such that $\alpha\leq\bar{\alpha}$. By definition, we know that $\bar{\alpha}$ is a $P$-cosmall root, and hence also a $B$-cosmall root. We also know that $\alpha$ is a $B$-cosmall root by \cite[Theorem~4.5(2)]{quasi-homogeneity-rev}. Because $\alpha$ and $\bar{\alpha}$ are both $B$-cosmall, the relation $\alpha\leq\bar{\alpha}$ implies that $\alpha^\vee\leq\bar{\alpha}^\vee$ by \cite[Lemma~4.7(a)]{curvenbhd}, which in turn implies that $d=\alpha^\vee+\mathbb{Z}\Delta_P^\vee=\bar{\alpha}^\vee+\mathbb{Z}\Delta_P^\vee$. This means that the greedy decomposition of $d$ consists of the unique element $\bar{\alpha}$. It follows from \cite[Theorem~6.16]{quasi-homogeneity-rev} that $\mathcal{B}_{R,e}\setminus R_P^+=\{\bar{\alpha}\}$. In other words, this means that $\alpha=\bar{\alpha}$. Thus, $\alpha$ is $P$-cosmall because $\bar{\alpha}$ is by definition.
\end{proof}

\begin{lem}
\label{lem:ass}

Let $d\in\Pi_P$ and assume that $(G,P,d)\neq(G_2,P_1,d_{G_2/P_1})$. Let $e$ be the lifting of $d$. For all $\alpha\in\mathcal{B}_{R,e}\setminus R_P^+$ and all $\gamma\in R_P^+$, we have $\left|(\gamma,\alpha^\vee)\right|\leq 2$.

\end{lem}

\begin{proof}

Let the notation be as in the statement. To prove this lemma, we can clearly assume that 
\begin{enumerate}
    \item 
    $d\neq 0$,
    \item\label{item:card>1}
    $\left|\mathcal{B}_{R,e}\right|>1$,
    \item\label{item:not_long}
    $\mathcal{B}_{R,e}\setminus R_P^+$ does not consist entirely of long roots,
    \item\label{item:not_B}
    $P\neq B$,
    \item\label{item:not_G2}
    $R$ is of type $\mathsf{G}_2$.
\end{enumerate}
Indeed, if $d=0$, then $e\in\mathbb{Z}\Delta_P^\vee$ by \cite[Fact~6.5(3)]{quasi-homogeneity-rev} and thus $\mathcal{B}_{R,e}\subseteq R_P^+$. If $\mathcal{B}_{R,e}=\varnothing$, the statement is empty. If $\mathcal{B}_{R,e}$ consists of a unique element, the assertion follows from Lemma~\ref{lem:prel_ass},~\ref{lem:prel_ass'} and the first assumption $d\neq 0$. If $\mathcal{B}_{R,e}\setminus R_P^+$ consists entirely of long roots, the assertion follows from \cite[Lemma~7.14]{quasi-homogeneity-rev}. If $P=B$, then $R_P=\varnothing$ and the statement is empty. Finally, the statement is obvious if $R$ is not of type $\mathsf{G}_2$.

Let us now assume all items in the enumerate above. By Item~\eqref{item:card>1},\eqref{item:not_G2} and Corollary~\ref{cor:less_than_d_X}, we know that $\left|\mathcal{B}_{R,e}\right|=\left|\mathcal{B}_R\right|=2$. Since $R$ is non simply laced, we know by Example~\ref{ex:center} that $w_o\in Z(W)$. Altogether, we see that Lemma~\ref{lem:equivalent_to_cascade} applies. We conclude that $e=d_{G/B}$, and thus $d=d_{G/P}$ by \cite[Example~6.3, Fact~6.5(3)]{quasi-homogeneity-rev} based on \cite{minimal}. Finally Item~\eqref{item:not_long},\eqref{item:not_B},\eqref{item:not_G2} and the previous reasoning imply that $(G,P,d)=(G_2,P_1,d_{G_2/P_1})$. But this case was excluded by assumption directly in the beginning.
\end{proof}

\begin{rem}

For $d\in\Pi_P$ such that $(G,P,d)=(G_2,P_1,d_{G_2/P_1})$, Lemma~\ref{lem:ass} clearly fails, i.e.\ there exist $\alpha\in\mathcal{B}_{R,e}\setminus R_P^+$ and $\gamma\in R_P^+$ such that $\left|(\gamma,\alpha^\vee)\right|=3$, where $e$ is the lifting of $d$.

\end{rem}

\begin{lem}[{Refinement of \cite[Lemma~7.15]{quasi-homogeneity-rev}}]
\label{lem:lem_7.15}

Let $d\in\Pi_P$. Let $e$ be the lifting of $d$. Then we have the following equality:
\[
-\sum_{\alpha\in\mathcal{B}_{R,e}\setminus R_P^+}\sum_{\gamma\in R_P^+\setminus I(s_\alpha)}(\gamma,\alpha^\vee)={}
\begin{aligned}[t]
&\operatorname{card}\{(\alpha,\gamma)\in(\mathcal{B}_{R,e}\setminus R_P^+)\times R_P^+\mid(\gamma,\alpha^\vee)=-1\}\\
{}+2&\operatorname{card}\{(\alpha,\gamma)\in(\mathcal{B}_{R,e}\setminus R_P^+)\times R_P^+\mid(\gamma,\alpha^\vee)=-2\}\\
{}+3&\operatorname{card}\{(\alpha,\gamma)\in(\mathcal{B}_{R,e}\setminus R_P^+)\times R_P^+\mid(\gamma,\alpha^\vee)=-3\}\,.
\end{aligned}
\]

\end{lem}

\begin{proof}

Let $d$ and $e$ be as in the statement. The proof of this lemma is easy and follows among the same lines as the explanations given in the \cite[first paragraph of the proof of Lemma~7.15]{quasi-homogeneity-rev}. Indeed, it suffices to consider the equivalence
\[
\alpha\in\mathcal{B}_{R,e}\setminus R_P^+,\gamma\in R_P^+\colon\gamma\notin I(s_\alpha)\Longleftrightarrow (\alpha,\gamma)\leq 0
\]
and the fact that for $\alpha$ and $\gamma$ as in the index set of the double sum in the statement summands with $(\gamma,\alpha^\vee)=0$ can be discarded.
\end{proof}

\begin{cor}
\label{cor:lem_7.15}

Let $d\in\Pi_P$ and assume that $(G,P,d)\neq(G_2,P_1,d_{G_2/P_1})$. Let $e$ be the lifting of $d$. Then we have the following equality:
\[
-\sum_{\alpha\in\mathcal{B}_{R,e}\setminus R_P^+}\sum_{\gamma\in R_P^+\setminus I(s_\alpha)}(\gamma,\alpha^\vee)={}
\begin{aligned}[t]
&\operatorname{card}\{(\alpha,\gamma)\in(\mathcal{B}_{R,e}\setminus R_P^+)\times R_P^+\mid(\gamma,\alpha^\vee)<\hphantom{-}0\}\\
{}+\hphantom{2}&\operatorname{card}\{(\alpha,\gamma)\in(\mathcal{B}_{R,e}\setminus R_P^+)\times R_P^+\mid(\gamma,\alpha^\vee)<-1\}\,.
\end{aligned}
\]

\end{cor}

\begin{proof}

This follows directly by combining Lemma~\ref{lem:ass} and Lemma~\ref{lem:lem_7.15}.
\end{proof}

\begin{thm}[{Refinement of \cite[Theorem~7.16]{quasi-homogeneity-rev} -- key inequality}]
\label{thm:key}

Let $d\in\Pi_P$ and assume that $(G,P,d)\neq(G_2,P_1,d_{G_2/P_1})$. Then we have the following inequality:
\[
(c_1(X),d)-\ell(z_d^P)\leq\operatorname{card}\bigl(\mathrm{TD}_{P,d}\sqcup\widetilde{\mathrm{TD}}_{P,d}\bigr)\,.
\]

\end{thm}

\begin{proof}

Let $d$ be as in the statement. Let $e$ be the lifting of $d$. With the help of the previous results, we compute
\begin{align*}
(c_1(X),d)+\ell(z_d^P)&=-\sum_{\alpha\in\mathcal{B}_{R,e}\setminus R_P^+}\sum_{\gamma\in R_P^+\setminus I(s_\alpha)}(\gamma,\alpha^\vee)+\operatorname{card}(\mathcal{B}_{R,e}\setminus R_P^+)&\text{by \cite[Lemma~6.25]{quasi-homogeneity-rev}}\\
&=\hphantom{+}\operatorname{card}\{(\alpha,\gamma)\in(\mathcal{B}_{R,e}\setminus R_P^+)\times R_P^+\mid(\gamma,\alpha^\vee)<\hphantom{-}0\}&\\
&\hphantom{={}}+\operatorname{card}\{(\alpha,\gamma)\in(\mathcal{B}_{R,e}\setminus R_P^+)\times R_P^+\mid(\gamma,\alpha^\vee)<-1\}&\\
&\hphantom{={}}+\operatorname{card}(\mathcal{B}_{R,e}\setminus R_P^+)
&\text{by Corollary~\ref{cor:lem_7.15}}\\
&\leq\hphantom{+}\operatorname{card}\bigl(\mathrm{TD}_{P,d}\bigr)+\operatorname{card}\bigl(\widetilde{\mathrm{TD}}_{P,d}\bigr)&\mathllap{\text{by Corollary~\ref{cor:tildetd_bij} and Theorem~\ref{thm:refinement_7.12}}}\\
&=\hphantom{+}\operatorname{card}\bigl(\mathrm{TD}_{P,d}\sqcup\widetilde{\mathrm{TD}}_{P,d}\bigr)&\mathllap{\text{by Lemma~\ref{lem:TDtilde}\eqref{item:not_in_TD} and Definition~\ref{def:tildeTD}.}\hphantom{\qed}}\qedhere
\end{align*}
\end{proof}

\begin{ex}
\label{ex:fail-key-G2}

Let $d\in\Pi_P$ and assume that $(G,P,d)=(G_2,P_1,d_{G_2/P_1})$. Let $\beta_1$ be the simple short root and $\beta_2$ the simple long root with the labeling of the simple roots as in \cite[Plate~IX]{bourbaki_roots}. Analogously as in Notation~\ref{not:chain_in_G2}, we define distinctive roots in $R^+\setminus R_P^+$ by the equations
\[
\theta_1=3\beta_1+2\beta_2\,,\,\theta_2=\beta_1\,.
\]
We then have by Notation~\ref{notation:cascade}, similarly as in Notation~\ref{not:chain_in_G2}, that
\[
\mathcal{B}_R\setminus R_P^+=\mathcal{B}_R=\{\theta_1,\theta_2\}
\]
and further
\begin{gather*}
\mathrm{TD}_{P,d}=\{-\theta_1,-\theta_2,-\theta_2-\beta_2\}=\{-\beta_1,-\beta_1-\beta_2,-3\beta_1-2\beta_2\}\,,\\
\widetilde{\mathrm{TD}}_{P,d}=\{-\theta_1+\beta_2\}=\{-3\beta_1-\beta_2\}\,.
\end{gather*}
From the last gather, we conclude that
\[
(c_1(X),d)-\ell(z_d^P)=10-5=5>4=3+1=\operatorname{card}\bigl(\mathrm{TD}_{P,d}\sqcup\widetilde{\mathrm{TD}}_{P,d}\bigr)\,,
\]
i.e.\ that the inequality in Theorem~\ref{thm:key} fails for the excluded case. We see that the assumption $(G,P,d)\neq(G_2,P_1,d_{G_2/P_1})$ is necessary and sufficient for the inequality in Theorem~\ref{thm:key} to hold.

\end{ex}

\section{Proof of quasi-homogeneity}
\label{sec:proof}

In this section, we give the proofs of Theorem~\ref{thm:main1},~\ref{thm:main2} which completely solve the question of quasi-homogeneity of $\overline{M}_{0,3}(X,d)$ under the action of $G$/$\operatorname{Aut}(X)$ for minimal degrees $d\in H_2(X)$. After the preliminary work done in the previous sections, it suffices to interpret the results, in particular the key inequality in Theorem~\ref{thm:key}, in more geometric terms.

\begin{references*}

The result of Theorem~\ref{thm:main1} was already anticipated in \cite{dmax}. Indeed, the authors clearly state that a detailed analysis will reveal that only one exception occurs for the quasi-homogeneity of $\overline{M}_{0,3}(X,d)$ under the action of $G$, namely the exception in $G_2/P_1$ identified in Assumption~\ref{ass:not_G2/P1}. In \cite[Commentaire~3.1]{dmax}, they say: \enquote{Dans ce paragraphe, on essaie de construire une courbe dont l'orbite est dense dans l'ensemble des coubres de degré $d$. Il se passe un phénomène bizarre : c'est toujours possible sauf pour $G_2/P_1$.} The reasoning which leads to Theorem~\ref{thm:main1} gives a precise meaning to this sentence.

\end{references*}

\begin{references*}

The automorphism group of $X$ was completely identified in \cite{demazure}. The result \cite[Théorème~1]{demazure} shows that in most cases, for example if $R$ is of type $\mathsf{F}_4$, we do not gain anything from the passage from $G$ to $\operatorname{Aut}(X)$. In all those cases, this tells us that we have to produce sufficient additional tangent directions to prove Theorem~\ref{thm:main1} as it was done in Section~\ref{sec:additional}. If $(G,P,d)=(G_2,P_1,d_{G_2/P_1})$, however, there is a crucial difference between $G$ and $\operatorname{Aut}(X)$ which we exploit in Theorem~\ref{thm:main2}. That we should use $\operatorname{Aut}(X)$ to handle quasi-homogeneity of the moduli space for this exception, was communicated to the author of this article by Nicolas Perrin in 2015. 

\end{references*}

\begin{notation}

We denote by $\mathfrak{g},\mathfrak{t},\mathfrak{b},\mathfrak{p}$ the Lie algebra of $G,T,B,P$ respectively. Then, the following holds:
\begin{itemize}
    \item 
    The Lie algebra $\mathfrak{g}$ is a complex simple Lie algebra.
    \item
    The Lie algebra $\mathfrak{t}$ is a Cartan subalgebra of $\mathfrak{g}$.
    \item
    The root system $R$ is the root system associated to $\mathfrak{g}$ and $\mathfrak{t}$.
    \item
    The Lie algebra $\mathfrak{b}$ is the Borel subalgebra of $\mathfrak{g}$ containing $\mathfrak{t}$ corresponding to the set of positive roots $R^+$.
    \item
    The Lie algebra $\mathfrak{p}$ is the standard parabolic subalgebra of $\mathfrak{g}$ with respect to $\mathfrak{b}$ with set of simple roots $\Delta_P$.
\end{itemize}
Furthermore, for a root $\alpha\in R$, we denote by $\mathfrak{g}_\alpha$ the root space associated to $\alpha$.

\end{notation}

\begin{notation}

For a root $\alpha\in R$, we denote by $U_\alpha$ the associated root group as defined in \cite[Theorem~26.3(a)]{humphreys_groups}.

\end{notation}

\begin{notation}

To simplify notation, we write $R(P)=R^+\cup R_P$ for short. The set of roots $R(P)$ is precisely the set of roots $\alpha\in R$ such that $U_\alpha\subseteq P$.

\end{notation}

\begin{notation}

For a Weyl group element $z\in W$, we denote by $P^z$ the conjugate of $P$, i.e.\ $P^z=zPz^{-1}$.

\end{notation}

\begin{lem}[{Refinement of \cite[Lemma~7.10]{quasi-homogeneity-rev}}]
\label{lem:lem_7.10}

Let $d\in\Pi_P$. Let $e$ be the lifting of $d$. Let $z=z_d^P$ for short. For all $\gamma\in R_P^+$, we have
\begin{enumerate}
    \item\label{item:lem_7.10} 
    $U_{-\gamma}\subseteq P\cap P^z$,
    \item\label{item:lem_7.10_refine}
    $\mathrlap{U_{\gamma'}}\hphantom{U_{-\gamma}}\subseteq P\cap P^z$ where $\gamma'=-z_e^B(\gamma)$.
\end{enumerate}

\end{lem}

\begin{proof}

Item~\eqref{item:lem_7.10} is a special case of \cite[Lemma~7.10]{quasi-homogeneity-rev}. Let the notation be as in the statement. To prove Item~\eqref{item:lem_7.10_refine}, it suffices by the arguments in the \cite[proof of Lemma~7.10]{quasi-homogeneity-rev} to show that $\gamma'\in R(P)$ and $zw_P(\gamma')\in R(P)$ which is in view of \cite[Fact~6.5(1)]{quasi-homogeneity-rev} equivalent to $\gamma'\in R(P)$ and $z_e^B(\gamma')\in R(P)$. But this latter statement is implied by $\gamma'\in R^+$ and $z_e^B(\gamma')\in R_P^-$ which is the content of Lemma~\ref{lem:TDtilde}\eqref{item:in_L}.
\end{proof}

\begin{rem}
\label{rem:tangent}

As in \cite[Proposition~1.1]{timashev}, we identify from now on the tangent space of $X$ at $1P$ with $\mathfrak{g}/\mathfrak{p}$.

\end{rem}

\begin{notation}

Let $d$ be a degree in $H_2(X)$. The moduli space $\overline{M}_{0,3}(X,d)$ comes equipped with three evaluation maps. For each $i\in\{1,2,3\}$, the $i$\textsuperscript{th} evaluation map $\mathrm{ev}_i\colon\overline{M}_{0,3}(X,d)\to X$ is defined by 
\[
\mathrm{ev}_i([C,p_1,p_2,p_3,\mu\colon C\to X])=\mu(p_i)\,.
\]

\end{notation}

\begin{notation}
\label{not:M(2)}


Let $d\in\Pi_P$. Let $z=z_d^P$ for short. We denote by $\overline{M}_{0,3}(X,d)(2)$ the fiber of the total evaluation map $\mathrm{ev}_1\times\mathrm{ev}_2\colon\overline{M}_{0,3}(X,d)\to X\times X$ over the point $(1P,zP)$. Note that $\overline{M}_{0,3}(X,d)(2)$ carries an action of $P\cap P^z$ induced by the action of $G$ on $\overline{M}_{0,3}(X,d)$.

\end{notation}

\begin{notation}
\label{notation:TfPd}

Let $d\in\Pi_P$. Let $z=z_d^P$ for short. Recall the definition of the morphism $f_{P,d}$ from Definition~\ref{def:dia} which will be in use onwards in this section. Recall from Remark~\ref{rem:in_M(2)} that $f_{P,d}$ is an element of $\overline{M}_{0,3}(X,d)(2)$. We denote by $T_{f_{P,d}}$ the tangent space at $f_{P,d}$ of the orbit
\[
(P\cap P^z)f_{P,d}\subseteq\overline{M}_{0,3}(X,d)(2)
\]
of $f_{P,d}$ under the action of $P\cap P^z$ on $\overline{M}_{0,3}(X,d)(2)$. As usual, we identify $T_{f_{P,d}}$ with a vector subspace of $\mathfrak{g}/\mathfrak{p}$ (Remark~\ref{rem:tangent}). As well as $\overline{M}_{0,3}(X,d)(2)$, the vector subspace $T_{f_{P,d}}$ carries an action of $P\cap P^z$.



\end{notation}

\begin{lem}
\label{lem:prel_main1}

Let $d\in\Pi_P$. We have an inclusion of vector subspaces of $\mathfrak{g}/\mathfrak{p}$ given by
\begin{gather*}
\textstyle{\bigoplus_\alpha(\mathfrak{g}_\alpha+\mathfrak{p})/\mathfrak{p}}\subseteq T_{f_{P,d}}\\
\text{where $\alpha$ runs through }
\mathrm{TD}_{P,d}\sqcup\widetilde{\mathrm{TD}}_{P,d}\,.
\end{gather*}

\end{lem}

\begin{proof}

Let $d\in\Pi_P$. Let $e$ be the lifting of $d$. Note that the sum in the statement of the lemma is actually direct because $\mathrm{TD}_{P,d}\sqcup\widetilde{\mathrm{TD}}_{P,d}\subseteq R^-\setminus R_P^-$ by Lemma~\ref{lem:TDtilde}\eqref{item:not_in_TD} and Definition~\ref{def:tildeTD}. By definition and the \cite[proof of the second claim in the proof of Theorem~8.2]{quasi-homogeneity-rev}, we already know that we have inclusions
\[
\bigoplus_{\alpha\in\mathcal{B}_{R,e}\setminus R_P^+}(\mathfrak{g}_{-\alpha}+\mathfrak{p})/\mathfrak{p}\subseteq\bigoplus_{\alpha\in\mathrm{TD}_{P,d}}(\mathfrak{g}_\alpha+\mathfrak{p})/\mathfrak{p}\subseteq T_{f_{P,d}}\,.
\]
Let $(\alpha',\gamma')$ be associated to $(\alpha,\gamma)$ where $\alpha\in\mathcal{B}_{R,e}\setminus R_P^+$ and $\gamma\in R_P^+$ are such that $(\gamma,\alpha^\vee)<-1$. By Lemma~\ref{lem:TDtilde}\eqref{item:in_tangent}, we have $-\alpha'+\gamma'\in\widetilde{\mathrm{TD}}_{P,d}$. To finish the proof of the lemma, it suffices to show that $(\mathfrak{g}_{-\alpha'+\gamma'}+\mathfrak{p})/\mathfrak{p}\subseteq T_{f_{P,d}}$. By the above displayed inclusion and by definition, we already know that $(\mathfrak{g}_{-\alpha'}+\mathfrak{p})/\mathfrak{p}\subseteq T_{f_{P,d}}$. Since $P\cap P^z$ where $z=z_d^P$ acts on $T_{f_{P,d}}$ by Notation~\ref{notation:TfPd}, we know that $U_{\gamma'}$ and thus $\mathfrak{g}_{\gamma'}$ act on $T_{f_{P,d}}$ by Lemma~\ref{lem:lem_7.10}\eqref{item:lem_7.10_refine}. We conclude that
\[
\left[\mathfrak{g}_{\gamma'},(\mathfrak{g}_{-\alpha'}+\mathfrak{p})/\mathfrak{p}\right]=\left(\left[\mathfrak{g}_{-\alpha'},\mathfrak{g}_{\gamma'}\right]+\mathfrak{p}\right)/\mathfrak{p}=(\mathfrak{g}_{-\alpha'+\gamma'}+\mathfrak{p})/\mathfrak{p}\subseteq T_{f_{P,d}}\,.\qedhere
\]
\end{proof}

\begin{cor}
\label{cor:prel_main1}

Let $d\in\Pi_P$ and assume that $(G,P,d)\neq(G_2,P_1,d_{G_2/P_1})$. We have the inequality
\[
(c_1(X),d)-\ell(z_d^P)\leq\dim(T_{f_{P,d}})\,.
\]

\end{cor}

\begin{proof}

This follows directly from Theorem~\ref{thm:key} and Lemma~\ref{lem:prel_main1}.
\end{proof}

\begin{thm}[{Refinement of \cite[Theorem~8.2]{quasi-homogeneity-rev}}]
\label{thm:main1}

Let $d\in\Pi_P$. The morphism $f_{P,d}$ has a dense open orbit in $\overline{M}_{0,3}(X,d)$ under the action of $G$ if and only if the moduli space $\overline{M}_{0,3}(X,d)$ is quasi-homogeneous under the action of $G$ if and only if $(G,P,d)\neq(G_2,P_1,d_{G_2/P_1})$.

\end{thm}

\begin{proof}

Suppose first that $(G,P,d)=(G_2,P_1,d_{G_2/P_1})$. We prove that $\overline{M}_{0,3}(X,d)$ is not quasi-homogeneous under the action of $G$. Indeed, this is the case because we have the inequality
\[
\dim\left(\overline{M}_{0,3}(X,d)\right)=(c_1(X),d)+\dim(X)=10+5=15>14=2+2\cdot 6=\dim(G)
\]
by \cite[Theorem~2(i)]{pand}. To complete the proof, we may assume that $(G,P,d)\neq(G_2,P_1,d_{G_2/P_1})$ and we have to prove that $f_{P,d}$ has a dense open orbit in $\overline{M}_{0,3}(X,d)$ under the action of $G$. As it was shown in the \cite[first four paragraphs of the proof of Theorem~8.2]{quasi-homogeneity-rev}, the inequality in Corollary~\ref{cor:prel_main1} is sufficient to achieve this.
\end{proof}

\begin{notation}

We denote by $L$ the Levi factor of $P$. Furthermore, we denote by $\mathfrak{l}$ the Lie algebra of $L$. With this notation, $\mathfrak{l}$ is the Levi subalgebra of $\mathfrak{p}$, and $R_P$ is the root system associated to $L$ and $T$, or $\mathfrak{l}$ and $\mathfrak{t}$. 

\end{notation}

\begin{lem}
\label{lem:levi}

Assume that $w_o(R_P)=R_P$. Then, $L=P\cap P^{w_o}$.

\end{lem}

\begin{proof}

By definition, it suffices to prove that
\[
R_P=\{\gamma\in R(P)\mid w_o(\gamma)\in R(P)\}\,.
\]
Because $w_o(R^+)=R^-$ and $w_o(R_P)=R_P$ by assumption, we have $w_o(R^+\setminus R_P^+)=R^-\setminus R_P^-$. This shows the inclusion \enquote{$\supseteq$} in the equation above. The inclusion \enquote{$\subseteq$} follows directly from the assumption.
\end{proof}

\begin{ex}
\label{ex:levi}

The assumption of Lemma~\ref{lem:levi} is for example satisfied if one of the following items holds:
\begin{enumerate}
    \item\label{item:w_o=-1}
    $w_o=-1$, e.g.\ if $R$ is non simply laced.
    \item\label{item:w_o_restricts}
    $\Delta_P$ is given by the support of $\alpha$ for some $\alpha\in\mathcal{B}_R$.
\end{enumerate}
Indeed, Item~\eqref{item:w_o=-1} is immediately clear from Example~\ref{ex:center}. Suppose that $\Delta_P$ is given by the support of $\alpha$ for some $\alpha\in\mathcal{B}_R$. By \cite[Proposition~1.10]{kostant}, we know that $w_o$ restricted to $\mathbb{R}\Delta_P$ is given by the longest element of $W_P$. Hence, it is clear that $w_o(R_P)=R_P$. This proves Item~\eqref{item:w_o_restricts}.

\end{ex}

\begin{thm}[{\cite[Remark~7.5]{quasi-homogeneity-rev}}]
\label{thm:main2}

Let $d\in\Pi_P$. The morphism $f_{P,d}$ has a dense open orbit in $\overline{M}_{0,3}(X,d)$ under the action of $\operatorname{Aut}(X)$. In particular, the moduli space $\overline{M}_{0,3}(X,d)$ is quasi-homogeneous under the action of $\operatorname{Aut}(X)$.

\end{thm}

\begin{rem}
\label{rem:appendix}

We will use the whole notation and the main results from Appendix~\ref{app:G2} in the proof of Theorem~\ref{thm:main2}. The proof of Theorem~\ref{thm:main2} is the only instance in the main body of this paper where this happens. The relevant commentary and explanations concerning Remark~\ref{rem:notation},~\ref{rem:notation2} will follow.

\end{rem}

\begin{proof}[Proof of Theorem~\ref{thm:main2}]

Let $d\in\Pi_P$. Because $f_{P,d}$ has a dense open orbit in $\overline{M}_{0,3}(X,d)$ under the action of $\operatorname{Aut}(X)$ if $f_{P,d}$ has a dense open orbit in $\overline{M}_{0,3}(X,d)$ under the action of $G$, we may, in view of Theorem~\ref{thm:main1}, assume from now on that $(G,P,d)=(G_2,P_1,d_{G_2/P_1})$. Since the conclusion of the theorem does only depend on the isomorphism class of $G$, we may further assume that $G$ is given by $G_2$ where $G_2$ is the specific instance of a connected, simply connected, simple, complex, linear algebraic group of type $\mathsf{G}_2$ constructed in Appendix~\ref{app:G2}. Since two parabolic subgroups with the same set of simple roots are conjugated, and since the conclusion of the theorem does only depend on the conjugacy class of $P$, we may also assume that $P$ is given by $P_1$ where $P_1$ is the maximal parabolic subgroup of $G_2$ constructed in the appendix. As a consequence of the assumption $G=G_2$, $P=P_1$, we have with the whole notation from the appendix the identities
\begin{itemize}
    \item 
    $T$ and $B$ are given as in the appendix,\footnote{Cf.~Remark~\ref{rem:notation}.\label{note:rem:notation}}
    \item
    $L=L_1$,
    \item
    $\mathfrak{t}$ and $\mathfrak{b}$ are given as in the appendix,\cref{note:rem:notation}
    \item
    $\mathfrak{g}=\mathfrak{g}_2$, $\mathfrak{p}=\mathfrak{p}_1$, $\mathfrak{l}=\mathfrak{l}_1$,
    \item
    $R=S$, $\Delta=\Pi$, $R^+=S^+$, $R^-=S^-$, 
    \item
    $R_P=S_{P_1}$, $\Delta_P=\Pi_{P_1}$,\footnote{The set $\Pi_P$ of all minimal degrees in $H_2(X)$ is no longer in use from now on until the end of the proof of Theorem~\ref{thm:main2}. The only minimal degree we have to consider in this proof is $d=d_{G_2/P_1}$ as in the next item. Hence, we can say in the annotated equation that $\Delta_P$ is given by $\Pi_{P_1}$ where $\Pi_{P_1}$ is defined as in Appendix~\ref{app:G2} (cf.~Remark~\ref{rem:notation2}).} $R_P^+=S_{P_1}^+$, $R_P^-=S_{P_1}^-$,
    \item
    $d=d_{G_2/P_1}$.\footnote{The last item as well as the identities $G=G_2$, $P=P_1$ are supposed to explain the notation in Assumption~\ref{ass:not_G2/P1} which is modeled for this proof.}
\end{itemize}

We may from now on and for the rest of this proof also make use of the inclusion $G_2\subseteq B_3$ and of the objects and symbols attached to the situation in $B_3$ introduced in Appendix~\ref{app:G2}. In particular, the identification $G_2/P_1=B_3/\tilde{P}_1$ in Remark~\ref{rem:identification} gives us an action of $B_3$ on $X$, and thus an action of $B_3$ on $\overline{M}_{0,3}(X,d)$ given by translation. To prove that $f_{P,d}$ has a dense open orbit in $\overline{M}_{0,3}(X,d)$ under the action of $\operatorname{Aut}(X)$, it clearly suffices to show that $f_{P,d}$ has a dense open orbit in $\overline{M}_{0,3}(X,d)$ under the action of $B_3$. We rather prove this latter statement after some preliminary observations which follow now.

\begin{notation}

We denote by $\tilde{W}$ the Weyl group associated to $B_3$ and $\tilde{T}$. We denote by $\tilde{w}_o$ the longest element of $\smash{\tilde{W}}$.

\end{notation}

\begin{fact}
\label{fact:w_o_and_w_o-tilde}

The element $\tilde{w}_o\in\tilde{W}$ considered as an automorphism of $\tilde{\mathfrak{t}}_{\mathbb{R}}$ restricts to an automorphism of $\mathfrak{t}_{\mathbb{R}}$ which is given by $w_o\in W$.

\end{fact}

\begin{proof}

Since $R$ and $\tilde{R}$ are non simply laced, we know by Example~\ref{ex:center} that $w_o=-1$ and $\tilde{w}_o=-1$ considered as automorphisms of $\mathfrak{t}_{\mathbb{R}}$ and $\smash{\tilde{\mathfrak{t}}_{\mathbb{R}}}$ respectively. The fact follows from this and the definition of the inclusion of vector spaces $\mathfrak{t}_{\mathbb{R}}\subseteq\smash{\tilde{\mathfrak{t}}_{\mathbb{R}}}$.
\end{proof}

\begin{cor}
\label{cor:t-fixed}

Under the identification $G_2/P_1=B_3/\tilde{P}_1$ as in Remark~\ref{rem:identification}, the $T$-fixed point $w_oP$ identifies with the $\tilde{T}$-fixed point $\tilde{w}_o\tilde{P}_1$.

\end{cor}

\begin{proof}

This follows directly from Fact~\ref{fact:w_o_and_w_o-tilde} and Corollary~\ref{cor:inclusions_groups}.
\end{proof}

\begin{fact}
\label{fact:degree_identification}

Under the identification $G_2/P_1=B_3/\tilde{P}_1$ as in Remark~\ref{rem:identification}, the degree $d\in H_2(X)$ identifies with the degree $d_{B_3/\tilde{P}_1}\in H_2(B_3/\tilde{P}_1)$.

\end{fact}

\begin{proof}

Under the identification $G_2/P_1=B_3/\tilde{P}_1$ a point certainly identifies with a point. Since, by Remark~\ref{rem:dX_point}, $d$ is the unique minimal degree in the quantum product of two point classes in $H^*(X)$, and $d_{B_3/\tilde{P}_1}$ is the unique minimal degree in the quantum product of two point classes in $\smash{H^*(B_3/\tilde{P}_1)}$, we see that $d$ identifies with $d_{B_3/\tilde{P}_1}$. This fact can also be seen more directly by explicit computation using Notation~\ref{notation:cascade}. Indeed, with the identification as in \cite[Convention~1.7]{minimal}, we have $d=d_{B_3/\tilde{P}_1}=2$.
\end{proof}

\begin{cor}
\label{cor:M(2)}

Under the identification $G_2/P_1=B_3/\tilde{P}_1$ as in Remark~\ref{rem:identification}, we have further identifications
\[
\overline{M}_{0,3}(X,d)=\overline{M}_{0,3}\bigl(B_3/\tilde{P}_1,d_{B_3/\tilde{P}_1}\bigr)\text{ and }\overline{M}_{0,3}(X,d)(2)=\overline{M}_{0,3}\bigl(B_3/\tilde{P}_1,d_{B_3/\tilde{P}_1}\bigr)(2)\,.
\]

\end{cor}

\begin{proof}

Note that $1P$ identifies with $1\tilde{P}_1$ under the identification $G_2/P_1=B_3/\tilde{P}_1$ because the inclusion of groups $G_2\subseteq B_3$ certainly preserves the identity element and because of Corollary~\ref{cor:inclusions_groups}. This together with Notation~\ref{not:M(2)}, Corollary~\ref{cor:t-fixed}, Fact~\ref{fact:degree_identification} yields the desired result.
\end{proof}

We return to the proof of Theorem~\ref{thm:main2} now. Note first that $L=P\cap P^{w_o}$ and $\tilde{L}_1=\tilde{P}_1\cap\tilde{P}_1^{\tilde{w}_o}$ by Lemma~\ref{lem:levi}, Example~\ref{ex:levi}\eqref{item:w_o=-1} because both $R$ and $\tilde{R}$ are non simply laced. Furthermore, we have an inclusion of groups $L\subseteq\tilde{L}_1$ by Corollary~\ref{cor:inclusions_groups}. By Notation~\ref{not:M(2)}, the moduli space $\overline{M}_{0,3}(X,d)(2)$ therefore naturally carries an action of $L$. Corollary~\ref{cor:M(2)} shows that this action extends to an action of the even larger group $\smash{\tilde{L}_1}$. To show that $f_{P,d}$ has a dense open orbit in $\overline{M}_{0,3}(X,d)$ under the action of $B_3$, it clearly suffices to show that $f_{P,d}$ has a dense open orbit in $\overline{M}_{0,3}(X,d)(2)$ under the action of $\tilde{L}_1$. We rather prove this latter statement.

To this end, let $\tilde{T}_{f_{P,d}}$ be the tangent space at $f_{P,d}$ of the orbit $\tilde{L}_1f_{P,d}\subseteq\overline{M}_{0,3}(X,d)(2)$ of $f_{P,d}$ under the action of $\smash{\tilde{L}_1}$ on $\overline{M}_{0,3}(X,d)(2)$. As usual, we identify $\smash{\tilde{T}_{f_{P,d}}}$ with a vector subspace of $\mathfrak{g}/\mathfrak{p}=\mathfrak{b}_3/\smash{\tilde{\mathfrak{p}}_1}$ (Remark~\ref{rem:tangent},~\ref{rem:identification}). As well as $\overline{M}_{0,3}(X,d)(2)$, the vector subspace $\smash{\tilde{T}_{f_{P,d}}}$ carries an action of $\smash{\tilde{L}_1}$. By means of derivation, this yields an action of $\smash{\tilde{\mathfrak{l}}_1}$ on $\smash{\tilde{T}_{f_{P,d}}}$ which extends to the action of $\smash{\tilde{\mathfrak{l}}_1}$ on the whole vector space $\mathfrak{g}/\mathfrak{p}$ defined in Remark~\ref{rem:identification}.
Let $\theta_1,\theta_2\in R^+\setminus R_P^+$ be defined as in Notation~\ref{not:chain_in_G2}. Recall that we have defined explicit root vectors $x_{-\theta_1},x_{-\theta_2}$ in Equations~\eqref{eq:g2->b3}. By definition of $f_{P,d}$, the tangent vector of $f_{P,d}$ at the point $1P$ is given by $x_{-\theta_1}+x_{-\theta_2}+\mathfrak{p}\in\mathfrak{g}/\mathfrak{p}$. By definition, we conclude that even $x_{-\theta_1}+x_{-\theta_2}+\mathfrak{p}\in\smash{\tilde{T}_{f_{P,d}}}$. Since $\smash{\tilde{\mathfrak{l}}_1}$ acts on $\smash{\tilde{T}_{f_{P,d}}}$, Lemma~\ref{lem:l1_acts_on_g2/p1} implies that
\[
\left[\tilde{\mathfrak{l}}_1,\left(\mathbb{C}\left(x_{-\theta_1}+x_{-\theta_2}\right)+\mathfrak{p}_1\right)/\mathfrak{p}_1\right]=\mathfrak{g}_2/\mathfrak{p}_1\subseteq\tilde{T}_{f_{P,d}}\,,
\]
and thus $\tilde{T}_{f_{P,d}}=\mathfrak{g}/\mathfrak{p}$. From this last equality, we follow that
\[
(c_1(X),d)-\ell(z_d^P)=10-5=5=\dim(\mathfrak{g}/\mathfrak{p})=\dim\bigl(\tilde{T}_{f_{P,d}}\bigr)\,.
\]
As it was shown in the \cite[first claim in the proof of Theorem~8.2]{quasi-homogeneity-rev}, this completes the proof of Theorem~\ref{thm:main2}.
\end{proof}

\appendix

\section{The inclusion of \texorpdfstring{$G_2$ into $B_3$}{G\texttwoinferior\ into B\textthreeinferior}}
\label{app:G2}

In this appendix, we define the inclusion $G_2\subseteq B_3$. We first define it on root vectors on the level of Lie algebras, and then pass to the associated groups. In the end, we need the explicit description of root vectors in $\mathfrak{g}_2$ and $\mathfrak{b}_3$ to verify the equation of vector spaces in Lemma~\ref{lem:l1_acts_on_g2/p1}. This result is then the crucial input for the proof of Theorem~\ref{thm:main2}.

\begin{references*}

There is an extensive literature on semisimple subalgebras of semisimple Lie algebras which goes back to Dynkin, cf.~\cite{selection_semisimple} for a selection. The way we define the embedding $G_2\subseteq B_3$ in this section is certainly not new. In fact, we used throughout the references \cite{knapp,m.postnikov} as a guide to define root vectors in $\mathfrak{g}_2$ and $\mathfrak{b}_3$, and modified the formulas whenever needed. Other literature which may does the same include \cite[Part~I]{yokota} and \cite{nato}. We were however not able to identify how our Equations~\eqref{eq:g2->b3} compare to those given in \cite[p.~241]{nato}.

\end{references*}

\begin{rem}
\label{rem:notation}

The notation introduced in this appendix is mostly independent from the notation in the main body of the text. In particular, we will define in the appendix symbols $T$, $B$ and $\mathfrak{t}$, $\mathfrak{b}$ which have a more specific meaning than they had in the main body. The only instance where both meanings are simultaneously in use is in the proof of Theorem~\ref{thm:main2}, and there we take care that they coincide by assumption.

\end{rem}

\begin{rem}
\label{rem:notation2}

Minimal degrees and the set $\Pi_P$ will not be subject of the considerations in this appendix. In particular, we will redefine the symbol $\Pi_{P_1}$ where $P_1$ is a parabolic subgroup of $G_2$ in a way which has nothing to do with the previous set $\Pi_P$ (even if $P=P_1$). We take care that no confusion arises from this double meaning.

\end{rem}

Let $\mathfrak{so}_7$ be the complex Lie algebra consisting of skew symmetric matrices of size $7\times 7$. We denote this Lie algebra by $\mathfrak{b}_3$ for short. Let $E_{i,j}$ be the matrix of size $7\times 7$ which has one as entry in the $i$\textsuperscript{th} row and the $j$\textsuperscript{th} column and zeros as entries elsewhere. We define elements of $\mathfrak{b}_3$ by the formula $E_{[i,j]}=E_{i,j}-E_{j,i}$. The following rules
\begin{gather*}
E_{[i,i]}=0\,,\,E_{[i,j]}=-E_{[j,i]}\,,\,\left[E_{[i,j]},E_{[j,k]}\right]=E_{[i,k]}\,,\\
\left[E_{[i,j]},E_{[k,l]}\right]=0\text{ if }\{i,j\}\cap\{k,l\}=\varnothing
\end{gather*}
allow to compute arbitrary commutators of $E_{[i,j]}$ and $E_{[k,l]}$. Note that the matrices $E_{[i,j]}$ where $1\leq i<j\leq 7$ form a basis of $\mathfrak{b}_3$. Let $\tilde{\mathfrak{t}}$ be the subspace of $\mathfrak{b}_3$ spanned by the elements $E_{[2,3]},E_{[4,5]},E_{[6,7]}$. The subspace $\tilde{\mathfrak{t}}$ is a Cartan subalgebra of $\mathfrak{b}_3$. Let $\varepsilon_1=iE_{[2,3]}$, $\varepsilon_2=iE_{[4,5]}$, $\varepsilon_3=iE_{[6,7]}$ for short. We consider the configuration
\[
\tilde{R}=\{\pm\varepsilon_i\pm\varepsilon_j\mid 1\leq i<j\leq 3\}\cup\{\pm\varepsilon_i\mid 1\leq i\leq 3\}
\]
inside the euclidean vector space $\tilde{\mathfrak{t}}_{\mathbb{R}}=\mathbb{R}\varepsilon_1\oplus\mathbb{R}\varepsilon_2\oplus\mathbb{R}\varepsilon_3$ endowed with the scalar product
\[
\left(\xi_1\varepsilon_1+\xi_2\varepsilon_2+\xi_3\varepsilon_3,\eta_1\varepsilon_1+\eta_2\varepsilon_2+\eta_3\varepsilon_3\right)=\xi_1\eta_1+\xi_2\eta_2+\xi_3\eta_3
\]
where $\xi_1,\xi_2,\xi_3,\eta_1,\eta_2,\eta_2\in\mathbb{R}$. The set $\tilde{R}$ is precisely the root system associated to $\mathfrak{b}_3$ and $\tilde{\mathfrak{t}}$. The root system $\tilde{R}$ is of type $\mathsf{B}_3$ as the notation suggests. We choose the base $\tilde{\Delta}$ of $\tilde{R}$ given by the simple roots
\[
\tilde{\beta}_1=\varepsilon_1-\varepsilon_2\,,\,\tilde{\beta}_2=\varepsilon_2-\varepsilon_3\,,\,\tilde{\beta}_3=\varepsilon_3\,.
\]
With this choice, the labeling of the simple roots as well as the explicit realization of the root system $\tilde{R}$ inside $\tilde{\mathfrak{t}}_{\mathbb{R}}$ 
is exactly as in \cite[Plate~II]{bourbaki_roots}. We denote the set of positive roots of $\tilde{R}$ with respect to $\tilde{\Delta}$ by $\tilde{R}^+$. For brevity, we denote by $\tilde{R}^-=-\tilde{R}^+$ the set of negative roots of $\tilde{R}$ with respect to $\tilde{\Delta}$.

We define elements of $\mathfrak{b}_3$ as follows
\allowdisplaybreaks\begin{align*}
\tilde{x}_{\tilde{\beta}_1}&=E_{[2,4]}+E_{[3,5]}+iE_{[2,5]}-iE_{[3,4]}\,,\\    
\tilde{x}_{\tilde{\beta}_2}&=E_{[4,6]}+E_{[5,7]}+iE_{[4,7]}-iE_{[5,6]}\,,\\
\tilde{x}_{\tilde{\beta}_3}&=E_{[1,6]}-iE_{[1,7]}\,,\\
\tilde{x}_{\tilde{\beta}_1+\tilde{\beta}_2}&=E_{[2,6]}+E_{[3,7]}+iE_{[2,7]}-iE_{[3,6]}\,,\\
\tilde{x}_{\tilde{\beta}_2+\tilde{\beta}_3}&=E_{[1,4]}-iE_{[1,5]}\,,\\
\tilde{x}_{\tilde{\beta}_1+\tilde{\beta}_2+\tilde{\beta}_3}&=E_{[1,2]}-iE_{[1,3]}\,,\\
\tilde{x}_{\tilde{\beta}_2+2\tilde{\beta}_3}&=E_{[4,6]}-E_{[5,7]}-iE_{[4,7]}-iE_{[5,6]}\,,\\
\tilde{x}_{\tilde{\beta}_1+\tilde{\beta}_2+2\tilde{\beta}_3}&=E_{[2,6]}-E_{[3,7]}-iE_{[2,7]}-iE_{[3,6]}\,,\\
\tilde{x}_{\tilde{\beta}_1+2\tilde{\beta}_2+2\tilde{\beta}_3}&=E_{[2,4]}-E_{[3,5]}-iE_{[2,5]}-iE_{[3,4]}\,.
\end{align*}\allowdisplaybreaks[0]%
We extend the definition of these elements to negative roots by setting $\tilde{x}_{-\tilde{\alpha}}=\overline{\tilde{x}_{\tilde{\alpha}}}$ for all $\tilde{\alpha}\in\tilde{R}^+$ where the bar denotes complex conjugation -- here and in what follows. We further write $(\mathfrak{b}_3)_{\tilde{\alpha}}=\mathbb{C}\tilde{x}_{\tilde{\alpha}}$ for all $\tilde{\alpha}\in\tilde{R}$. In \cite[Chapter~II, Section~1, Example~2]{knapp}, it was shown in general for type $\mathsf{B}$ and in particular for type $\mathsf{B}_3$ that $(\mathfrak{b}_3)_{\tilde{\alpha}}$ is the root space associated to $\tilde{\alpha}\in\tilde{R}$ and that we have the usual root space decomposition / Cartan decomposition given by
$
\mathfrak{b}_3=\tilde{\mathfrak{t}}\oplus\bigoplus_{\tilde{\alpha}\in\tilde{R}}(\mathfrak{b}_3)_{\tilde{\alpha}}
$.

Let us now consider the following subspaces
\begin{align*}
\mathfrak{t}&=\{\xi_1\varepsilon_1+\xi_2\varepsilon_2+\xi_3\varepsilon_3\text{ where $\xi_1,\xi_2,\xi_3\in\mathbb{C}$ such that $-\xi_1+\xi_2-\xi_3=0$}\}\,,\\
\mathfrak{t}_{\mathbb{R}}&=\{\xi_1\varepsilon_1+\xi_2\varepsilon_2+\xi_3\varepsilon_3\text{ where $\xi_1,\xi_2,\xi_3\in\mathbb{R}$ such that $-\xi_1+\xi_2-\xi_3=0$}\}
\end{align*}
of $\tilde{\mathfrak{t}}$ and $\tilde{\mathfrak{t}}_{\mathbb{R}}$ respectively. The latter subspace is endowed with an euclidean structure inherited from $\tilde{\mathfrak{t}}_{\mathbb{R}}$. We consider the following two elements
\[
\beta_1=\tfrac{1}{3}\varepsilon_1+\tfrac{2}{3}\varepsilon_2+\tfrac{1}{3}\varepsilon_3\,,\,\beta_2=-\varepsilon_2-\varepsilon_3
\]
of $\mathfrak{t}_{\mathbb{R}}$. The two elements $\beta_1$ and $\beta_2$ generate a root system $S$ of type $\mathsf{G}_2$ inside $\mathfrak{t}_{\mathbb{R}}$. We choose the base $\Pi$ of $S$ given by the simple roots $\beta_1,\beta_2$. With this choice, the labeling of the simple roots is as in \cite[Plate~IX]{bourbaki_roots}, i.e.\ $\beta_1$ is the simple short root and $\beta_2$ is the simple long root.\footnote{However, the explicit realization of the root system $S$ inside $\mathfrak{t}_{\mathbb{R}}$ is slightly different from the one in \cite[Plate~IX]{bourbaki_roots}. Both are of course isomorphic.} We denote the set of positive roots of $S$ with respect to $\Pi$ by $S^+$. For brevity, we denote by $S^-=-S^+$ the set of negative roots of $S$ with respect to $\Pi$.

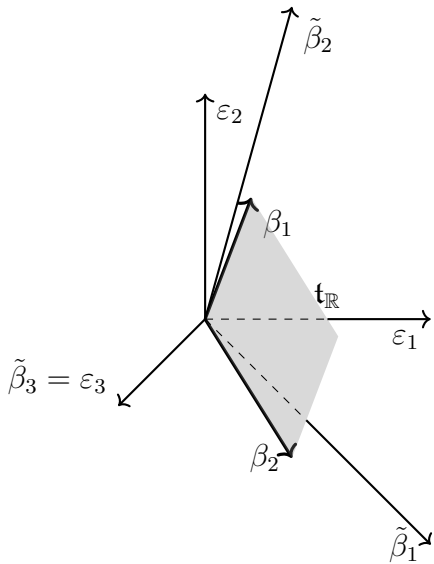
\begin{figure}
\centering
\begin{tikzpicture}
\draw[very thick,->] (0,0,0) coordinate(O) -- (0,-3,-3) coordinate(b2) node[anchor=east]{$\beta_2$};
\draw[very thick,->] (O) -- (1,2,1) coordinate(b1) node[anchor=north west]{$\beta_1$};
\fill[gray,opacity=0.3,name path=plane] (O) -- (b1) -- ++ (b2) 
node[very near end,above,black,opacity=1]{$\mathfrak{t}_{\mathbb{R}}$} -- (b2) -- cycle;
\path[name path=E1] (O) -- (3,0,0);
\draw[name intersections={of=plane and E1,by={aux,i1}},dashed] (O)--(i1);
\draw[thick,->] (i1) -- (3,0,0) node[anchor=north east]{$\varepsilon_1$};
\draw[thick,->] (O) -- (0,3,0) node[anchor=north west]{$\varepsilon_2$};
\draw[thick,->] (O) -- (0,0,3) node[anchor=south east]{$\tilde{\beta}_3=\varepsilon_3$};
\path[name path=B1] (O) -- (3,-3,0);
\draw[name intersections={of=plane and B1,by={aux,i2}},dashed] (O) -- (i2);
\draw[thick,->] (i2) -- (3,-3,0)  node[anchor=east]{$\tilde{\beta}_1$};
\draw[thick,->] (O) -- (0,3,-3) node[anchor=north west]{$\tilde{\beta}_2$};
\end{tikzpicture}
\caption{Illustration of simple roots.}
\label{fig:simple_roots}
\end{figure}

We define additional elements of $\mathfrak{b}_3$ as follows
\begin{equation}
\mathrlap{\;\,\left\{\vphantom{
\begin{aligned}
\hphantom{\tilde{x}_{\tilde{\beta}_1+2\tilde{\beta}_2+2\tilde{\beta}_3}}\mathllap{x_{\beta_1}}&=\mathrlap{2i\tilde{x}_{\tilde{\beta}_2+\tilde{\beta}_3}+\tilde{x}_{\tilde{\beta}_1+\tilde{\beta}_2+2\tilde{\beta}_3}}\,,\hphantom{E_{[2,4]}+E_{[3,5]}+iE_{[2,5]}-iE_{[3,4]}\,,}\\
x_{\beta_2}&=\tilde{x}_{-\tilde{\beta}_2-2\tilde{\beta}_3}\,,\\
x_{\beta_1+\beta_2}&=2\tilde{x}_{-\tilde{\beta}_3}+i\tilde{x}_{\tilde{\beta}_1}\,,\\
x_{2\beta_1+\beta_2}&=i\tilde{x}_{\tilde{\beta}_2}-2\tilde{x}_{\tilde{\beta}_1+\tilde{\beta}_2+\tilde{\beta}_3}\,,\\
x_{3\beta_1+\beta_2}&=\tilde{x}_{\tilde{\beta}_1+2\tilde{\beta}_2+2\tilde{\beta}_3}\,,\\
x_{3\beta_1+2\beta_2}&=\tilde{x}_{\tilde{\beta}_1+\tilde{\beta}_2}\,.
\end{aligned}}\right.}
\begin{aligned}
\hphantom{\tilde{x}_{\tilde{\beta}_1+2\tilde{\beta}_2+2\tilde{\beta}_3}}\mathllap{x_{\beta_1}}&=\mathrlap{2i\tilde{x}_{\tilde{\beta}_2+\tilde{\beta}_3}+\tilde{x}_{\tilde{\beta}_1+\tilde{\beta}_2+2\tilde{\beta}_3}}\,,\hphantom{E_{[2,4]}+E_{[3,5]}+iE_{[2,5]}-iE_{[3,4]}\,,}\\
x_{\beta_2}&=\tilde{x}_{-\tilde{\beta}_2-2\tilde{\beta}_3}\,,\\
x_{\beta_1+\beta_2}&=2\tilde{x}_{-\tilde{\beta}_3}+i\tilde{x}_{\tilde{\beta}_1}\,,\\
x_{2\beta_1+\beta_2}&=i\tilde{x}_{\tilde{\beta}_2}-2\tilde{x}_{\tilde{\beta}_1+\tilde{\beta}_2+\tilde{\beta}_3}\,,\\
x_{3\beta_1+\beta_2}&=\tilde{x}_{\tilde{\beta}_1+2\tilde{\beta}_2+2\tilde{\beta}_3}\,,\\
x_{3\beta_1+2\beta_2}&=\tilde{x}_{\tilde{\beta}_1+\tilde{\beta}_2}\,.
\end{aligned}\label{eq:g2->b3}
\end{equation}
We extend the definition of these elements to negative roots in $S^-$ by setting $x_{-\alpha}=\overline{x_\alpha}$ for all $\alpha\in S^+$. We further write $(\mathfrak{g}_2)_\alpha=\mathbb{C}x_\alpha$ for all $\alpha\in S$. Let us now define $\mathfrak{g}_2$ as the Lie subalgebra of $\mathfrak{b}_3$ generated by $x_{\beta_1},x_{\beta_2},x_{-\beta_1},x_{-\beta_2}$. With this notation fixed, it was shown in \cite[Lecture~14, Proposition~1]{m.postnikov} that the following items hold.
\begin{itemize}
    \item 
    The Lie algebra $\mathfrak{g}_2$ is the complex simple Lie algebra of type $\mathsf{G}_2$.
    \item
    The Lie algebra $\mathfrak{t}$ is a Cartan subalgebra of $\mathfrak{g}_2$.
    \item
    The root system $S$ is the root system associated to $\mathfrak{g}_2$ and $\mathfrak{t}$.
    \item
    Each $(\mathfrak{g}_2)_\alpha$ is the root space associated to $\alpha\in S$.
    \item
    We have the root space decomposition / Cartan decomposition given by\newline 
    $
    \mathfrak{g}_2=\mathfrak{t}\oplus\bigoplus_{\alpha\in S}(\mathfrak{g}_2)_\alpha
    $.
\end{itemize}

\begin{rem}

To summarize, we record the formulas
\begin{gather*}
    \overline{\tilde{\alpha}}=-\tilde{\alpha}\,,\,\overline{\tilde{\alpha}^\vee}=-\tilde{\alpha}^\vee\,,\,\overline{\tilde{x}_{\tilde{\alpha}}}=\tilde{x}_{-\tilde{\alpha}}\text{ for all $\tilde{\alpha}\in\tilde{R}$,}\\
    \overline{\alpha}=-\alpha\,,\,\overline{\alpha^\vee}=-\alpha^\vee\,,\,\overline{x_{\alpha}}=x_{-\alpha}\text{ for all $\alpha\in S$.}
\end{gather*}

\end{rem}

\allowdisplaybreaks\begin{notation}

We define the following sets of roots
\begin{align*}
&\tilde{R}_{\tilde{P}_1}&&\mathrlap{\text{the root subsystem of $\tilde{R}$ generated by $\tilde{\beta}_2$ and $\tilde{\beta}_3$},}\hphantom{\text{the Levi factor of $\tilde{P}_1$ with Lie algebra $\tilde{\mathfrak{l}}_1$ such that $\tilde{R}_{\tilde{P}_1}$}}\\
&\tilde{\Delta}_{\tilde{P}_1}&&\text{the base of $\tilde{R}_{\tilde{P}_1}$ given by the simple roots $\tilde{\beta}_2,\tilde{\beta}_3$,}\\
&\tilde{R}_{\tilde{P}_1}^+&&\text{the set of positive roots of $\tilde{R}_{\tilde{P}_1}$ with respect to $\tilde{\Delta}_{\tilde{P}_1}$,}\\
&\tilde{R}_{\tilde{P}_1}^-&&\text{the set of negative roots of $\tilde{R}_{\tilde{P}_1}$ with respect to $\tilde{\Delta}_{\tilde{P}_1}$.}\\
\intertext{For brevity, we set}
&&&S_{P_1}=\{\pm\beta_2\}\,,\,S_{P_1}^+=\Pi_{P_1}=\{\beta_2\}\,,\,S_{P_1}^-=\{-\beta_2\}\,.
\end{align*}
These lastly defined sets of roots have the analogous interpretations with respect to $S$ as the one in the previous align with respect to $\tilde{R}$. We define the following infinitesimal objects
\begin{align*}
    \tilde{\mathfrak{b}}&=\textstyle{\mathrlap{\mspace{2mu}\tilde{\mathfrak{t}}}\hphantom{\tilde{\mathfrak{b}}}\oplus\bigoplus_{\tilde{\alpha}\in\tilde{R}^+_{\hphantom{\tilde{P}_1}}}(\mathfrak{b}_3)_{\tilde{\alpha}}}&&
    \begin{tabular}{@{}l@{}}
    the Borel subalgebra of $\mathfrak{b}_3$ containing $\tilde{\mathfrak{t}}$ corre-\\
    sponding to the set of positive roots $\tilde{R}^+$,
    \end{tabular}\\
    \tilde{\mathfrak{p}}_1&=\textstyle{\tilde{\mathfrak{b}}\oplus\bigoplus_{\tilde{\alpha}\in\tilde{R}_{\tilde{P}_1}^-}(\mathfrak{b}_3)_{\tilde{\alpha}}}&&
    \begin{tabular}{@{}l@{}}
    the standard parabolic subalgebra of $\tilde{\mathfrak{b}}_3$ with\\
    respect to $\tilde{\mathfrak{b}}$ with set of simple roots $\tilde{\Delta}_{\tilde{P}_1}$,
    \end{tabular}\\
    \tilde{\mathfrak{l}}_1&=\textstyle{\mathrlap{\mspace{2mu}\tilde{\mathfrak{t}}}\hphantom{\tilde{\mathfrak{b}}}\oplus\bigoplus_{\tilde{\alpha}\in\tilde{R}_{\tilde{P}_1}}(\mathfrak{b}_3)_{\tilde{\alpha}}}&&
    \begin{tabular}{@{}l@{}}
    the Levi subalgebra of $\tilde{\mathfrak{p}}_1$ such that $\tilde{R}_{\tilde{P}_1}$ is the\\
    root system associated to $\tilde{\mathfrak{l}}_1$ and $\tilde{\mathfrak{t}}$,
    \end{tabular}\\
    \mathfrak{b}&=\textstyle{\mathrlap{\mspace{2mu}\mathfrak{t}}\hphantom{\tilde{\mathfrak{b}}}\oplus\bigoplus_{\alpha\in S^+_{\hphantom{P_1}}}(\mathfrak{g}_2)_{\alpha}}&&
    \begin{tabular}{@{}l@{}}
    the Borel subalgebra of $\mathfrak{g}_2$ containing $\mathfrak{t}$ corre-\\
    sponding to the set of positive roots $S^+$,
    \end{tabular}\\
    \mathfrak{p}_1&=\textstyle{\mathfrak{b}\oplus\bigoplus_{\alpha\in S_{P_1}^-}(\mathfrak{g}_2)_{\alpha}}&&
    \begin{tabular}{@{}l@{}}
    the standard parabolic subalgebra of $\mathfrak{g}_2$ with\\ 
    respect to $\mathfrak{b}$ with set of simple roots $\Pi_{P_1}$,
    \end{tabular}\\
    \mathfrak{l}_1&=\textstyle{\mathrlap{\mspace{2mu}\mathfrak{t}}\hphantom{\tilde{\mathfrak{b}}}\oplus\bigoplus_{\alpha\in S_{P_1}}(\mathfrak{g}_2)_{\alpha}}&&
    \begin{tabular}{@{}l@{}}
    the Levi subalgebra of $\mathfrak{p}_1$ such that $S_{P_1}$ is the\\
    root system associated to $\mathfrak{l}_1$ and $\mathfrak{t}$.
    \end{tabular}    
\end{align*}
By integration, we define further objects
\addtocounter{footnote}{1}
\footnotetext{This group is usually known as $\mathrm{Spin}_7$.}
\addtocounter{footnote}{-1}
\begin{align*}
    &B_3&&
    \begin{tabular}{@{}l@{}}
    the connected, simply connected, simple, complex,\\
    linear algebraic group of type $\mathsf{B}_3$ with Lie algebra $\mathfrak{b}_3$,\footnotemark
    \end{tabular}\\
    &\tilde{T}&&
    \begin{tabular}{@{}l@{}}
    the maximal torus inside $B_3$ with Lie algebra $\tilde{\mathfrak{t}}$ such\\
    that $\tilde{R}$ is the root system associated to $B_3$ and $\tilde{T}$,
    \end{tabular}\\
    &\tilde{B}&&
    \begin{tabular}{@{}l@{}}
    the Borel subgroup of $B_3$ containing $\tilde{T}$ with Lie alge-\\
    bra $\tilde{\mathfrak{b}}$ corresponding to the set of positive roots $\tilde{R}^+$,
    \end{tabular}\\
    &\tilde{P}_1&&
    \begin{tabular}{@{}l@{}}
    the standard parabolic subgroup of $B_3$ with respect\\
    to $\tilde{B}$ with Lie algebra $\tilde{\mathfrak{p}}_1$ and set of simple roots $\tilde{\Delta}_{\tilde{P}_1}$,\\
    \end{tabular}\\ 
    &\tilde{L}_1&&
    \begin{tabular}{@{}l@{}}
    the Levi factor of $\tilde{P}_1$ with Lie algebra $\tilde{\mathfrak{l}}_1$ such that $\tilde{R}_{\tilde{P}_1}$\\
    is the root system associated to $\tilde{L}_1$ and $\tilde{T}$,
    \end{tabular}\\
    &G_2&&
    \begin{tabular}{@{}l@{}}
    the connected, simply connected, simple, complex,\\
    linear algebraic group of type $\mathsf{G}_2$ with Lie algebra $\mathfrak{g}_2$,
    \end{tabular}\\
    &T&&
    \begin{tabular}{@{}l@{}}
    the maximal torus inside $G_2$ with Lie algebra $\mathfrak{t}$ such\\
    that $S$ is the root system associated to $G_2$ and $T$,
    \end{tabular}\\
    &B&&
    \begin{tabular}{@{}l@{}}
    the Borel subgroup of $G_2$ containing $T$ with Lie alge-\\
    bra $\mathfrak{b}$ corresponding to the set of positive roots $S^+$,
    \end{tabular}\\
    &P_1&&
    \begin{tabular}{@{}l@{}}
    the standard parabolic subgroup of $G_2$ with respect\\
    to $B$ with Lie algebra $\mathfrak{p}_1$ and set of simple roots $\Pi_{P_1}$,\\
    \end{tabular}\\ 
    &L_1&&
    \begin{tabular}{@{}l@{}}
    the Levi factor of $P_1$ with Lie algebra $\mathfrak{l}_1$ such that $S_{P_1}$\\
    is the root system associated to $L_1$ and $T$.
    \end{tabular}
\end{align*}

\end{notation}\allowdisplaybreaks[0]

\begin{rem}

Note that $\mathfrak{p}_1$ and $\tilde{\mathfrak{p}}_1$ are both maximal parabolic subalgebras of $\mathfrak{g}_2$ and $\mathfrak{b}_3$, and that $P_1$ and $\tilde{P}_1$ are both maximal parabolic subgroups of $G_2$ and $B_3$, respectively.

\end{rem}

\begin{lem}
\label{lem:inclusions_Lie}

We have inclusions and  equalities of Lie algebras
\[
\mathfrak{g}_2\subseteq\mathfrak{b}_3\,,\,\mathfrak{t}\subseteq\tilde{\mathfrak{t}}\,,\,\mathfrak{p}_1\subseteq\tilde{\mathfrak{p}}_1\,,\,\mathfrak{t}=\mathfrak{g}_2\cap\tilde{\mathfrak{t}}\,,\,\mathfrak{p}_1=\mathfrak{g}_2\cap\tilde{\mathfrak{p}}_1\,.
\]

\end{lem}

\begin{proof}

The first two inclusions of Lie algebras follow directly from the constructions above. Taking into account the second inclusion, the third one follows by definition and inspection of Equations~\eqref{eq:g2->b3}. From these inclusions, we infer the inclusions $\mathfrak{t}\subseteq\mathfrak{g}_2\cap\tilde{\mathfrak{t}}$, $\mathfrak{p}_1\subseteq\mathfrak{g}_2\cap\tilde{\mathfrak{p}}_1$. The first of these two latter inclusions must be an equality because, as a Cartan subalgebra, $\mathfrak{t}$ is a maximal abelian subalgebra of $\mathfrak{g}_2$ and because $\mathfrak{g}_2\cap\tilde{\mathfrak{t}}$ is itself an abelian subalgebra of $\mathfrak{g}_2$. Finally, concerning the inclusion $\mathfrak{p}_1\subseteq\mathfrak{g}_2\cap\tilde{\mathfrak{p}}_1$, if it would be strict, then $\mathfrak{g}_2=\mathfrak{g}_2\cap\tilde{\mathfrak{p}}_1$ and thus $\mathfrak{g}_2\subseteq\tilde{\mathfrak{p}}_1$ because $\mathfrak{p}_1$ is a maximal parabolic subalgebra of $\mathfrak{g}_2$ and because $\mathfrak{g}_2\cap\tilde{\mathfrak{p}}_1$ is itself a standard parabolic subalgebra of $\mathfrak{g}_2$ with respect to $\mathfrak{b}$. We have however $x_{-3\beta_1-2\beta_2}=\tilde{x}_{-\tilde{\beta}_1-\tilde{\beta}_2}\in\mathfrak{g}_2\setminus\tilde{\mathfrak{p}}_1$.
\end{proof}

\begin{cor}
\label{cor:inclusion_Levi}

We have an inclusion and an equality of Lie algebras 
\[
\mathfrak{l}_1\subseteq\tilde{\mathfrak{l}}_1\,,\,\mathfrak{l}_1=\mathfrak{g}_2\cap\tilde{\mathfrak{l}}_1\,.
\]

\end{cor}

\begin{proof}

The inclusion follows from the equality in the statement of the lemma. The equality $\mathfrak{l}_1=\mathfrak{g}_2\cap\tilde{\mathfrak{l}}_1$ follows from the equality $\mathfrak{t}=\mathfrak{g}_2\cap\tilde{\mathfrak{t}}$ in Lemma~\ref{lem:inclusions_Lie} and the equality of vector spaces
\[
(\mathfrak{g}_2)_{-\beta_2}\oplus(\mathfrak{g}_2)_{\beta_2}=\bigoplus_{\alpha\in S}(\mathfrak{g}_2)_\alpha\cap\bigoplus_{\tilde{\alpha}\in\tilde{R}_{\tilde{P}_1}}(\mathfrak{b}_3)_{\tilde{\alpha}}
\]
which can be easily inferred from Equations~\eqref{eq:g2->b3}.
\end{proof}

\begin{cor}
\label{cor:iso_of_tangent_spaces}

We have an isomorphism $\mathfrak{g}_2/\mathfrak{p}_1\cong\mathfrak{b}_3/\tilde{\mathfrak{p}}_1$ of vector spaces induced by the inclusion $\mathfrak{g}_2\subseteq\mathfrak{b}_3$.

\end{cor}

\begin{proof}

Lemma~\ref{lem:inclusions_Lie} shows that the inclusion $\mathfrak{g}_2\subseteq\mathfrak{b}_3$ induces an injective homomorphism $\mathfrak{g}_2/\mathfrak{p}_1\hookrightarrow\mathfrak{b}_3/\tilde{\mathfrak{p}}_1$ of vector spaces. But this homomorphism must be an isomorphism because
\begin{align*}
\dim\left(\mathfrak{g}_2/\mathfrak{p}_1\right)&=\mathrlap{\operatorname{card}\left(S^-\setminus S_{P_1}^-\right)}\hphantom{\operatorname{card}\bigl(\tilde{R}^-\setminus\tilde{R}_{\tilde{P}_1}^-\bigr)}=5\,,\\
\dim\left(\mathfrak{b}_3/\tilde{\mathfrak{p}}_1\right)&=\operatorname{card}\bigl(\tilde{R}^-\setminus\tilde{R}_{\tilde{P}_1}^-\bigr)=5\,.\qedhere
\end{align*}
\end{proof}

\begin{cor}
\label{cor:inclusions_groups}

We have inclusions and equalities of linear algebraic groups
\[
G_2\subseteq B_3\,,\,T\subseteq\tilde{T}\,,\,P_1\subseteq\tilde{P}_1\,,\,L_1\subseteq\tilde{L}_1\,,\,T=G_2\cap\tilde{T}\,,\,P_1=G_2\cap\tilde{P}_1\,,\,L_1=G_2\cap\tilde{L}_1\,.
\]

\end{cor}

\begin{proof}

This corollary follows by definition of the linear algebraic groups in question and integration of the inclusions and equalities of Lie algebras in Lemma~\ref{lem:inclusions_Lie} and Corollary~\ref{cor:inclusion_Levi}.
\end{proof}

\begin{cor}
\label{cor:iso_of_groups}

We have an isomorphism $G_2/P_1\cong B_3/\tilde{P}_1$ of algebraic varieties induced by the inclusion $G_2\subseteq B_3$.\footnote{The projective variety $G_2/P_1\cong B_3/\tilde{P}_1$ is known to be a quadric of dimension five in $\mathbb{P}_{\mathbb{C}}^6$ \cite[p.~924]{agricola}.}

\end{cor}

\begin{proof}

Corollary~\ref{cor:inclusions_groups} shows that the inclusion $G_2\subseteq B_3$ induces an injective morphism $G_2/P_1\hookrightarrow B_3/\tilde{P}_1$ of algebraic varieties. But this morphism must be an isomorphism because both sides $G_2/P_1$ and $B_3/\tilde{P}_1$ are irreducible and by Corollary~\ref{cor:iso_of_tangent_spaces} of the same dimension.
\end{proof}

\begin{rem}

Note that the Borel subalgebras $\mathfrak{b}$ and $\tilde{\mathfrak{b}}$ are not preserved under the inclusion $\mathfrak{g}_2\subseteq\mathfrak{b}_3$. Indeed, we have $\mathfrak{b}\not\subseteq\tilde{\mathfrak{b}}$, e.g., because $\smash{x_{\beta_2}=\tilde{x}_{-\tilde{\beta}_2-2\tilde{\beta}_3}\in\mathfrak{b}\setminus\tilde{\mathfrak{b}}}$. Consequently, the Borel subgroups $B$ and $\tilde{B}$ are also not preserved under the inclusion $G_2\subseteq B_3$, i.e.\ we have $B\not\subseteq\tilde{B}$.

\end{rem}

\begin{rem}
\label{rem:identification}

From now on, we identify the objects related by the isomorphisms in Corollary~\ref{cor:iso_of_tangent_spaces},~\ref{cor:iso_of_groups}, i.e.\ we set $\mathfrak{g}_2/\mathfrak{p}_1=\mathfrak{b}_3/\tilde{\mathfrak{p}}_1$, $G_2/P_1=B_3/\tilde{P}_1$. By means of this identification, we get an action of $\mathfrak{p}_1$, $\mathfrak{l}_1$ and also additionally of $\tilde{\mathfrak{p}}_1$, $\tilde{\mathfrak{l}}_1$ on $\mathfrak{g}_2/\mathfrak{p}_1$. In a similar vein, we get an action of $B_3$ and of all of its subgroups, e.g.\ of $G_2$, on $G_2/P_1$. We will freely use these actions from now on, for example in Lemma~\ref{lem:l1_acts_on_g2/p1} and its proof.

\end{rem}

\begin{notation}
\label{not:chain_in_G2}

We define distinctive roots in $S^+\setminus S_{P_1}^+$ by the equations
\[
\theta_1=3\beta_1+2\beta_2\,,\,\theta_2=\beta_1\,.
\]
By Notation~\ref{notation:cascade}, we then have
\[
\mathcal{B}_S\setminus S_{P_1}^+=\mathcal{B}_S=\{\theta_1,\theta_2\}\,.
\]

\end{notation}

\begin{lem}
\label{lem:l1_acts_on_g2/p1}

With the action of $\tilde{\mathfrak{l}}_1$ defined on $\mathfrak{g}_2/\mathfrak{p}_1$ as in Remark~\ref{rem:identification}, we have the following equality of vector spaces
\[
\left[\tilde{\mathfrak{l}}_1,\left(\mathbb{C}\left(x_{-\theta_1}+x_{-\theta_2}\right)+\mathfrak{p}_1\right)/\mathfrak{p}_1\right]=\mathfrak{g}_2/\mathfrak{p}_1\,.
\]

\end{lem}

\begin{proof}

By letting $\mathfrak{t}$ inside $\tilde{\mathfrak{t}}$ act and because $\theta_1$ and $\theta_2$ are (strongly) orthogonal, we first see that
\[
\left[\tilde{\mathfrak{t}},\mathbb{C}\left(x_{-\theta_1}+x_{-\theta_2}\right)\right]\supseteq\left[\mathfrak{t},\mathbb{C}\left(x_{-\theta_1}+x_{-\theta_2}\right)\right]=\mathbb{C}x_{-\theta_1}+\mathbb{C}x_{-\theta_2}
\]
where the last sum is actually direct. Using this inclusion,
we compute with the help of the definition of the respective root vectors in $\mathfrak{g}_2$ and $\mathfrak{b}_3$ (cf.~Equations~\eqref{eq:g2->b3}) that inside $\mathfrak{b}_3$ the following inclusions of vector subspaces hold
\begin{align*}
\mathfrak{b}_3\supseteq\smash{\left[\tilde{\mathfrak{l}}_1,\mathbb{C}\left(x_{-\theta_1}+x_{-\theta_2}\right)\right]}+\tilde{\mathfrak{p}}_1\supseteq{}&\mathbb{C}x_{-\theta_1}+\mathbb{C}x_{-\theta_2}+\\
&\mathrlap{\mathbb{C}\left[\tilde{x}_{\tilde{\beta}_2\hphantom{-}},x_{-\theta_1}+x_{-\theta_2}\right]}\hphantom{\mathbb{C}\left[\tilde{x}_{-\tilde{\beta}_2-\tilde{\beta}_3},x_{-\theta_1}+x_{-\theta_2}\right]}+\mathrlap{\mathbb{C}\left[\tilde{x}_{\tilde{\beta}_3\hphantom{-}},x_{-\theta_1}+x_{-\theta_2}\right]}\hphantom{\mathbb{C}\left[\tilde{x}_{-\tilde{\beta}_2-2\tilde{\beta}_3},x_{-\theta_1}+x_{-\theta_2}\right]{}}+\\
&\mathbb{C}\left[\tilde{x}_{\tilde{\beta}_2+\tilde{\beta}_3\hphantom{-}},x_{-\theta_1}+x_{-\theta_2}\right]+\mathbb{C}\left[\tilde{x}_{\tilde{\beta}_2+2\tilde{\beta}_3\hphantom{-}},x_{-\theta_1}+x_{-\theta_2}\right]+\\
&\mathrlap{\mathbb{C}\left[\tilde{x}_{-\tilde{\beta}_2},x_{-\theta_1}+x_{-\theta_2}\right]}\hphantom{\mathbb{C}\left[\tilde{x}_{-\tilde{\beta}_2-\tilde{\beta}_3},x_{-\theta_1}+x_{-\theta_2}\right]}+\mathrlap{\mathbb{C}\left[\tilde{x}_{-\tilde{\beta}_3},x_{-\theta_1}+x_{-\theta_2}\right]}\hphantom{\mathbb{C}\left[\tilde{x}_{-\tilde{\beta}_2-2\tilde{\beta}_3},x_{-\theta_1}+x_{-\theta_2}\right]{}}+\\
&\mathbb{C}\left[\tilde{x}_{-\tilde{\beta}_2-\tilde{\beta}_3},x_{-\theta_1}+x_{-\theta_2}\right]+\mathbb{C}\left[\tilde{x}_{-\tilde{\beta}_2-2\tilde{\beta}_3},x_{-\theta_1}+x_{-\theta_2}\right]+\tilde{\mathfrak{p}}_1\displaybreak[0]\\
\supseteq{}&\mathbb{C}\tilde{x}_{-\tilde{\beta}_1-\tilde{\beta}_2}+\mathbb{C}\tilde{x}_{-\tilde{\beta}_1-\tilde{\beta}_2-2\tilde{\beta}_3}{}+{}\\
&\mathbb{C}\tilde{x}_{-\tilde{\beta}_1\hphantom{-\tilde{\beta}_2}}+\mathbb{C}\tilde{x}_{-\tilde{\beta}_1-\tilde{\beta}_2-\tilde{\beta}_3\hphantom{2}}+\mathbb{C}\tilde{x}_{-\tilde{\beta}_1-2\tilde{\beta}_2-2\tilde{\beta}_3}+\tilde{\mathfrak{p}}_1\displaybreak[0]\\
={}&\textstyle{\bigoplus_{\tilde{\alpha}\in\tilde{R}^-\setminus\tilde{R}_{\tilde{P}_1}^-}(\mathfrak{b}_3)_{\tilde{\alpha}}+\tilde{\mathfrak{p}}_1}\\
={}&\mathfrak{b}_3\,,
\end{align*}
where the very last sum is actually direct. In total, the previous align means that
\[
\left[\tilde{\mathfrak{l}}_1,\mathbb{C}\left(x_{-\theta_1}+x_{-\theta_2}\right)\right]+\tilde{\mathfrak{p}}_1=\mathfrak{b}_3\,.
\]
If we plug this equality into the following computation, we find in view of the identification in Remark~\ref{rem:identification} the desired result:
\[
\left[\tilde{\mathfrak{l}}_1,\left.\left(\mathbb{C}\left(x_{-\theta_1}+x_{-\theta_2}\right)+\mathfrak{p}_1\right)\right/\mathfrak{p}_1\right]=\left.\left(\left[\tilde{\mathfrak{l}}_1,\mathbb{C}\left(x_{-\theta_1}+x_{-\theta_2}\right)\right]+\tilde{\mathfrak{p}}_1\right)\right/\tilde{\mathfrak{p}}_1=\mathfrak{b}_3/\tilde{\mathfrak{p}}_1=\mathfrak{g}_2/\mathfrak{p}_1\,.\qedhere
\]
\end{proof}

\begin{ex}

With the help of the explicit formulas in Example~\ref{ex:fail-key-G2}, we see that
\begin{gather*}
\textstyle{\mathfrak{g}_2/\mathfrak{p}_1=((\mathfrak{g}_2)_{-2\beta_1-\beta_2}+\mathfrak{p}_1)/\mathfrak{p}_1\oplus\bigoplus_{\alpha}((\mathfrak{g}_2)_\alpha+\mathfrak{p}_1)/\mathfrak{p}_1\,,}\\
\text{where $\alpha$ runs through }
\mathrm{TD}_{P_1,d_{G_2/P_1}}\sqcup\widetilde{\mathrm{TD}}_{P_1,d_{G_2/P_1}}\,,
\end{gather*}
i.e.\ that the vectors $x_\alpha+\mathfrak{p}_1$ where $\alpha$ runs though the aforementioned set of roots contained in $S^-\setminus S_{P_1}^-$ do not generate the whole vector space $\mathfrak{g}_2/\mathfrak{p}_1$ but a subspace of codimension one.

\end{ex}

\bibliographystyle{aomplain}
\bibliography{lib}

\end{document}